\documentclass{amsart}

\pdfoutput=1
  \pdfpageattr{/Group <</S /Transparency /I true /CS /DeviceRGB>>}

\usepackage{todonotes}
\makeatletter
\def\ps@pprintTitle{%
 \let\@oddhead\@empty
 \let\@evenhead\@empty
 \def\@oddfoot{}%
 \let\@evenfoot\@oddfoot}
\makeatother

\usepackage{xspace}
\usepackage{marvosym}
\usepackage{microtype}
\microtypesetup{tracking,kerning,spacing}
\microtypecontext{spacing=nonfrench}

\usepackage{hyperref}
\usepackage{amsfonts,amsmath,amssymb,amsthm}
\usepackage{graphicx}
\usepackage{tikz}
\usepackage{pgfplots}
\usepackage{caption}
\usepackage[labelfont=rm]{subcaption}

\usepackage{dsfont}
\usepackage{wasysym}

\usepackage{array}
\usepackage{float}
\usepackage{multirow}
\usepackage{stackrel}

\newtheorem{theorem}{Theorem}[section]
\newtheorem{lemma}[theorem]{Lemma}
\newtheorem{proposition}[theorem]{Proposition}
\newtheorem{corollary}[theorem]{Corollary}
\newcounter{mainTheorem}
\newtheorem{maintheorem}[mainTheorem]{Theorem}
\theoremstyle{definition}
\newtheorem{definition}[theorem]{Definition}

\theoremstyle{remark}
\newtheorem{remark}[theorem]{Remark}

\newtheorem{example}[theorem]{Example}

\usetikzlibrary{positioning}

\tikzset{EdgeStyle/.style = {color=black!60, thick}}

\tikzset{VertexStyle/.style = {
    circle, fill=white, draw=black,
    text           = black,
    inner sep      = 2pt,
    outer sep      = 0pt,
    minimum size   = 10 pt}}

\tikzset{BoxVertex/.style = {rectangle, fill=white, draw=black, text = black, inner sep = 2.5pt, outer sep = 0pt, minimum size = 10pt}}

\DeclareMathOperator{\supp}{supp}
\DeclareMathOperator{\conv}{conv}
\DeclareMathOperator{\cone}{cone}
\DeclareMathOperator{\sign}{sign}

\newcommand\RR{{\mathbb R}}
\newcommand\ZZ{{\mathbb Z}}

\newcommand\Lcal{\mathcal{L}}

\newcommand\ground{\mathsf{E}}
\newcommand\rows{\mathsf{R}}
\newcommand\augground{\widetilde{\mathsf{E}}}

\DeclareRobustCommand{\rchi}{{\mathpalette\irchi\relax}}
\newcommand{\irchi}[2]{\raisebox{\depth}{$#1\chi$}}
\newcommand{\ssimplex}{\triangle}

\global\long\def\ee{\mathbf{e}}

\global\long\def\supp{\mathrm{supp}}
\global\long\def\oneone{\mathbf{1}}
\global\long\def\zero{\mathbf{0}}
\global\long\def\conv{\mathrm{conv}}
\global\long\def\bd{\mathrm{bd}}
\global\long\def\lattice{\mathcal{L}}
\global\long\def\orientedmatroid{\mathcal{M}}

\global\long\def\crosspolytope{\diamondsuit}
\global\long\def\cube{\square}
\global\long\def\sign{\mathrm{sign}}
\global\long\def\poset{\mathtt{P}}
\global\long\def\qoset{\mathtt{Q}}
\global\long\def\fmixed{\mathtt{S}}
\global\long\def\str{\mathrm{star}}
\global\long\def\dual{\vee}
\global\long\def\neighbourhood{\mathcal{N}}
\global\long\def\ground{\mathsf{E}}
\global\long\def\rows{\mathsf{R}}
\global\long\def\posmap{\varphi}
\global\long\def\rank{\rho}
\global\long\def\norm#1{\left\Vert #1\right\Vert }

\global\long\def\eqclass#1{\tilde{#1}}
\global\long\def\eqrel{/\negmedspace\sim}
\global\long\def\uni{\bigcup}

\global\long\def\covectors{\mathcal{L}}

\global\long\def\cc{\mathbf{c}}
\global\long\def\octfmixed{\octagon_\fmixed}

\makeatletter
\def\input@path{{./IMG2/}{./}}
\makeatother

\graphicspath{{IMG2/}}

\title[Patchworking Oriented Matroids]{Patchworking Oriented Matroids}
\author{Marcel Celaya, Georg Loho, Chi Ho Yuen}

\address{Technische Universit\"at Berlin, Institut f\"ur Mathematik, Sekr. MA4-1, Stra{\ss}e des 17 Juni 136, D-10623 Berlin}
\email{celaya@math.tu-berlin.de}

\address{Department of Mathematics, London School of Economics and Political Science, London, WC2A 2AE, UK}
\email{g.loho@lse.ac.uk}

\address{Division of Applied Mathematics, Brown University, Providence, RI 02912, USA}
\email{Chi\_Ho\_Yuen@Brown.edu}

\subjclass[2020]{
  52C40; % Oriented matroids in discrete geometry
  05E45, % Combinatorial aspects of simplicial complexes
  14T15, % Combinatorial aspects of tropical varieties
  52C30, % Planar arrangements of lines and pseudolines
  57N60, % Cellularity in topological manifolds
  57Q99  % PL-topology - None of the above, but in this section
}


\begin{document}

% \dedicatory{}

\begin{abstract}
In a previous work, we gave a construction of (not necessarily realizable) oriented matroids from a triangulation of a product of two simplices.
In this follow-up paper, we use a variant of Viro's patchworking to derive a topological representation of the oriented matroid directly from the polyhedral structure of the triangulation, hence finding a combinatorial manifestation of patchworking besides tropical algebraic geometry.
We achieve this by rephrasing the patchworking procedure as a controlled cell merging process, guided by the structure of tropical oriented matroids. 
A key insight is a new promising technique to show that the final cell complex is regular.

%  We introduce a construction of oriented matroids from a triangulation of a product of two simplices.
%  This is the second part of our paper, focusing on the topological aspect of our construction.
%  In particular, we derive a topological representation of the oriented matroid from the polyhedral structure of the triangulation.
%  This relies on a variant of Viro's patchworking and insights on the structure of tropical oriented matroids. 
%  A recurring theme in our work is that various tropical constructions can be extended beyond tropicalization with new formulations and proof methods.
\end{abstract}

\maketitle

%To-do-list:
%\begin{itemize}
%\item Double check if any comment from Josephine Yu or Laura Anderson concern this part
%\item Address the question about flips in triangulation and the corresponding OMs.
%\item Elaborate a bit more with the relation/comparison with Horn (mention~\cite[Proposition 6.5]{Horn1}) and Hersh?
%\item Work out Jesus's non-regular example
%\item Explain why regularity is not obvious and require a (technical) proof
%\end{itemize}

\section{Introduction}

Oriented matroids are ordinary matroids equipped with extra sign data, which capture and extend the combinatorics of directed graphs, real hyperplane arrangements, and more generally linear dependence over $\RR$.
They appear in many subjects in mathematics and related areas, from discrete geometry and optimization algorithms to algebraic geometry and topology; we refer the reader to \cite[Chapter 1 \& 2]{BLSWZ:1993} for more examples.
Besides having various equivalent combinatorial axiom systems, a major result in the theory of oriented matroids is the {\em Topological Representation Theorem} of Folkman and Lawrence \cite{FolkmanLawrence:1978}, which states that every oriented matroid can be represented by a {\em pseudosphere arrangement} (a topological generalization of real hyperplane arrangements) and vice versa.

Inspired by the work of Sturmfels and Zelevinsky on maximal minors, and connections with tropical geometry and optimization problems, the authors of the current paper introduced in \cite{CelayaLohoYuen:2020} a construction of (uniform) oriented matroids from triangulations of $\ssimplex_{d-1}\times\ssimplex_{n-1}$ with suitable sign data.
While the case of regular triangulations is implicit in the literature using {\em signed tropicalization}, considering general triangulations allows us to obtain non-realizable oriented matroids, including Ringel's classical example (see \cite[Section~4.2]{CelayaLohoYuen:2020}).
The construction given in~\cite{CelayaLohoYuen:2020} expresses an oriented matroid using a chirotope, which assigns signs to the ordered bases of the underlying matroid.
More precisely, by encoding the cells of the triangulations by forests of the complete bipartite graph $K_{d,n}$, every basis is associated with a matching, and the sign of the basis is the sign of the matching as a permutation.
However, in view of the polyhedral structure of a triangulation, it is natural to ask whether the topological realization of such oriented matroids can be related to the triangulation directly.
In this paper, we provide a construction to achieve this.

To do so, we adapt the method of patchworking, which goes back to Viro in the 1980s~\cite{Viro:1983}.
Viro's method has numerous applications in real algebraic geometry and tropical geometry (see the survey by Viro \cite{Viro:2006}\footnote{The title of our paper is inspired by the title of this survey.}), and is related to the Gelfand--Kapranov--Zelevinsky theory~\cite{GelfandKapranovZelevinsky:1994}.
The idea of (combinatorial) patchworking is that one can construct piecewise linear objects isotopic to real algebraic varieties by some ``cut and paste'' procedure, starting with a {\em regular} subdivision of a Newton polytope with sign data.
Sturmfels used this idea in~\cite{Sturmfels:1994a} to study complete intersections, where mixed subdivisions play the crucial role to derive the structure of the intersecting hypersurfaces.
While the latter are focused on the study of the intersections, we deal with the whole cellular complex cut out by them. 
%only concerned with the structure of the intersection itself, we use this variant of patchworking to generalize the encoding of all possible intersections of linear hypersurfaces.

\begin{maintheorem} \label{thm:representation}
Given a fine mixed subdivision of $n\ssimplex_{d-1}$ and a sign matrix, we can construct a pseudosphere arrangement representing the oriented matroid in \cite{CelayaLohoYuen:2020} via a patchworking procedure.
\end{maintheorem}

From this patchworking procedure, we implicitly derive an abstract real phase structure in the sense of~\cite{BBR:2017,RenaudineauShaw:2019} from the interplay of the subdivision and the sign matrix. 
Since most works on patchworking aim to construct real algebro geometric objects, their proofs usually use {\em tropicalization} of polynomials or similar techniques.
In contrast, the aforementioned non-realizable example shows that we can produce non-algebro geometric objects.
This suggests that patchworking could be applied for other topological problems beyond tropicalization.

Our proof uses a combination of combinatorial and topological methods, and is loosely based on Horn's second Topological Representation Theorem for tropical oriented matroids \cite{Horn1}. Roughly speaking, we show that it is possible to interpolate between the dual complex of a patchworking complex, which may be regarded as a cell decomposition of the boundary of the sphere, and a pseudosphere arrangement representing our oriented matroid. This is done by carefully ``merging'' cells together, ensuring at each step that the combinatorics and the topology are controlled. A similar technique was used by Hersh in her work on total positivity \cite{Hersh:2014}. Actually, our results imply a ``Topological Representation Theorem'' for each interpolation step between Horn's result and the result of Folkman and Lawrence. 

We note the work of De Loera and Wicklin in~\cite{DeLoeraWicklin:1998} which extends and studies patchworking in dimension two giving rise to a combinatorial version of Hilbert's Lemma. Furthermore, Itenberg and Shustin derive for dimension two in~\cite{ItenbergShustin:2002} that patchworking with arbitrary subdivisions produces real pseudoholomorphic curves. However, it seems not much work on patchworking with general subdivisions has been done in higher dimension before us.

The paper is organized as follows.
In Section~\ref{sec:background}, we collect essential definitions and background for the central objects in this paper, as well as a summary of results from \cite{CelayaLohoYuen:2020} which are needed in this part.
Section~\ref{sec:patchwork} is devoted to stating the main theorem, Theorem~\ref{thm:representation}, and contains an illustration of the rank 3 case. Sections \ref{sec:elim_system} and \ref{sec:quotients_reg_cell_cpxs} elaborate on the two main ingredients in the proof of Theorem~\ref{thm:representation}, namely \emph{elimination systems} and \emph{quotients of regular cell complexes}. The main theorem itself is proved in Section~\ref{sec:pf_main_theorem}. The dependencies of each of the sections on each other are shown below:

\[
\begin{array}{ccccc}
 \ref{sec:OM}\text{-}\ref{sec:part+I} & \longrightarrow & \ref{sec:patchwork}\\
  &  &  & \searrow\\
 \ref{sec:poset_lattice_quotients} & \longrightarrow & \ref{sec:elim_system} & \longrightarrow & \ref{sec:pf_main_theorem}\\
 & \searrow &  & \nearrow\\
  &  & \ref{sec:quotients_reg_cell_cpxs}
\end{array}
\]

\section{Background} \label{sec:background}

%\todo[inline]{YCH: Again need to revise to avoid self-plagiarism.}

Throughout the paper, we fix a ground set $\ground$ of size $n$ and a set $\rows$ of size $d \leq n$.
We often identify $\ground,\rows$ with $[n] = \{1,2,\dots,n\}$ and $[d]$, hence fixing an ordering for them.
We use $\{+,-,0\}$ and $\{1,-1,0\}$ for signs interchangeably, and we adopt the ordering $+,->0$ of signs.
This is extended componentwise to a partial order on sign vectors. For a sign vector $X$ and a sign $s\in\{+,-,0\}$, we denote by $X^s$ the set of all indices $e$ such that $X_e = s$.

\subsection{Oriented Matroids} \label{sec:OM}

We refer the reader to \cite{BLSWZ:1993} for a comprehensive survey on oriented matroids.
%The following definition is used to describe oriented matroids in \cite{CelayaLohoYuen:2020}. \todo{GL: Is the last sentence really necessary? It sounds like we use an unusual definition in our part I. }

\begin{definition} \label{def:chirotope}
A {\em chirotope} on $\ground$ of rank $d$ is a non-zero, alternating map $\rchi:\ground^d\rightarrow\{+,-,0\}$ that satisfies the {\em Grassmann--Pl\"{u}cker (GP) relation}:\\
For any $x_1,\ldots,x_{d-1},y_1,\ldots, y_{d+1}\in \ground$, the $d+1$ expressions 
\begin{equation*}
(-1)^k \rchi(x_1,\ldots,x_{d-1},y_k)\rchi(y_1,\ldots,\widehat{y_k},\ldots,y_{d+1}), \qquad k=1,\ldots, d+1,
\end{equation*}
either contain both a positive and a negative term, or are all zeros.
Here $\widehat{y_k}$ means that we remove $y_k$ from the list. 

By the alternating property, we can specify a chirotope by its values over all $d$-tuples of strictly increasing elements, which are identified with $d$-subsets of $\ground$.
\end{definition}

%\begin{example} \label{ex:realizable+oriented+matroid}
% A chirotope is the generalization of the signs of maximal minors of a real matrix (oriented matroids coming from this way are said to be {\em realizable}).
%  More precisely, let $A \in \RR^{d \times n}$ be a rectangular matrix with rank~$d$.
%  Then
%  \[
%  \rchi(j_1,j_2,\dots,j_d) = \sign\det\left(a^{(j_1)},a^{(j_2)},\dots,a^{(j_d)}\right) ,
%  \]
%  where $a^{(j_k)}$ denotes columns of $A$, is the chirotope of an oriented matroid of rank $d$.
%\end{example}

For the purpose of topological constructions, we switch to an alternative axiom system.
We recall that the {\em composition} $X \circ Y$ of two signed vectors agrees with $X$ in all positions $e \in \ground$ with $X_e \neq 0$, and agrees with $Y$ otherwise.

\begin{definition} \label{def:covector}
A collection of sign vectors $\Lcal\subset\{+,-,0\}^\ground$ is the collection of {\em covectors} of an oriented matroid if
\begin{enumerate}
\item ${\bf 0}\in\Lcal$.
\item If $X\in\Lcal$, then $-X\in\Lcal$.
\item If $X,Y\in\Lcal$, then $X \circ Y\in\Lcal$.
\item For any $X,Y\in\Lcal$ and $e\in X^+\cap Y^-$, there exists $Z\in\Lcal$ such that $Z_e=0$, and $Z_f = (X \circ Y)_f =(Y\circ X)_f$ for all $f$ for which the latter equality holds.
%such that $Z^+\subset (X^+\cup Y^+)\setminus\{e\}$, $Z^-\subset (X^-\cup Y^-)\setminus\{e\}$, and $(\underline{X}\triangle\underline{Y})\cup(X^+\cap Y^+)\cup(X^-\cap Y^-)\subset\underline{Z}$.
\end{enumerate}
\end{definition}

\begin{example} \label{ex:realizable+oriented+matroid+covector}
Let $M$ be an oriented matroid realized by the real matrix $A \in \RR^{d \times n}$, i.e.,
\[
\rchi(j_1,j_2,\dots,j_d) = \sign\det\left(a^{(j_1)},a^{(j_2)},\dots,a^{(j_d)}\right).
\]
The covectors of $M$ are precisely the sign patterns of the vectors in the row space of $A$.
\end{example}

Finally, we give the definition of pseudosphere arrangements in the statement of the Topological Representation Theorem mentioned in the introduction.

\begin{definition} \label{def:pseudosphere+arrangement}
A {\em pseudosphere arrangement} of rank $d$  is a collection $(S_e:e\in \ground)$ of $(d-2)$-spheres piecewise-linearly (PL), central symmetrically embedded on $S^{d-1}$ together with sign data, i.e., for each $S_e$, specify a positive and a negative side for the two connected components of $S^{d-1}\setminus S_e$. 
Furthermore, we require that for any $\ground'\subset \ground$, $S_{\ground'}:=\bigcap_{e\in \ground'} S_e$ is also a PL sphere, and that for every other $S_e$, either $S_{\ground'}\subset S_e$ or $S_{\ground'}\cap S_e$ is a PL sphere of codimension 1 within $S_e$.
\end{definition}

The face lattice of such an arrangement is isomorphic to the {\em covector lattice} of the oriented matroid; we again refer the reader to \cite[Chapter~5]{BLSWZ:1993} for details.

\subsection{Triangulations of $\ssimplex_{d-1}\times\ssimplex_{n-1}$ and Polyhedral Matching Fields} \label{sec:triangulation}

We refer the reader to \cite{DeLoeraRambauSantos:2010}, respectively \cite{CelayaLohoYuen:2020,LohoSmith:2020}, for details in polyhedral geometry and matching fields.
We denote the $(k-1)$-simplex, respectively the product of a $(d-1)$-simplex and an $(n-1)$-simplex, by $\ssimplex_{k-1}$ and  respectively $\ssimplex_{d-1}\times\ssimplex_{n-1}$.
We fix their embeddings in $\RR^k$ (respectively $\mathbb{R}^d\times\mathbb{R}^n $) as $\conv\{{\bf e}_i \colon i \in [k]\}$ (respectively $\conv\{ ({\bf e}_i,{\bf e}_j) \colon i \in [d],j \in [n] \}$).
%We give a brief introduction to the necessary polyhedral notions and refer the reader to~\cite{DeLoeraRambauSantos:2010} for more details.
%For a more detailed introduction to matching fields, we refer to \cite{CelayaLohoYuen:2020}. 
%We denote the $(k-1)$-simplex by $\ssimplex_{k-1}$, which is embedded in $ \RR^k$ as $\conv\{{\bf e}_i \colon i \in [k]\}$.
%The product $\ssimplex_{d-1}\times\ssimplex_{n-1}$ of a $(d-1)$-simplex and an $(n-1)$-simplex is the convex hull
%\[
%\ssimplex_{d-1}\times\ssimplex_{n-1} = \conv\{ ({\bf e}_i,{\bf e}_j) \colon i \in [d],j \in [n] \}\subset\mathbb{R}^d\times\mathbb{R}^n \enspace .
%\]

A collection $\mathcal{T}$ of full-dimensional simplices is a \emph{triangulation} of $\ssimplex_{d-1}\times\ssimplex_{n-1}$  if
%A \emph{triangulation} of $\ssimplex_{d-1}\times\ssimplex_{n-1}$ is a collection of full-dimensional simplices $\mathcal{T}$ whose vertices form a subset of the vertices of $\ssimplex_{d-1}\times\ssimplex_{n-1}$, such that
\begin{enumerate}
\item the vertices of each simplex is a subset of the vertices of $\ssimplex_{d-1}\times\ssimplex_{n-1}$,
\item the union of all simplices in $\mathcal{T}$ is $\ssimplex_{d-1}\times\ssimplex_{n-1}$, 
\item the intersection of any two simplices in $\mathcal{T}$ is a common face of them.
\end{enumerate}

By identifying the vertices of $\ssimplex_{d-1}\times\ssimplex_{n-1}$ with the edges of the complete bipartite graph $K_{\rows, \ground}$, each full-dimensional simplex satisfying (1) gives rise to a spanning tree of $K_{\rows, \ground}$.
A combinatorial characterization of when a collection of trees forms a triangulation is given in \cite[Proposition 7.2]{ArdilaBilley:2007}.
%Each simplex in $\mathcal{T}$ gives rise to a spanning tree of the complete bipartite graph on the vertex set $\rows \sqcup \ground$ by identifying a vertex $({\bf e}_i,{\bf e}_j)$ with an edge, see \cite[Proposition 7.2]{ArdilaBilley:2007} for a combinatorial characterization of the collection of such trees.

\begin{definition}[{\cite[Section~2.3]{CelayaLohoYuen:2020}}] \label{ex:polyhedral+matching+field}
A \emph{polyhedral matching field} is the collection of all $\rows$-saturating matchings (those covering all nodes in $\rows$) that are subgraphs of the trees encoding a triangulation of $\ssimplex_{d-1}\times\ssimplex_{n-1}$, which consists of exactly one perfect matching $M_\sigma$ between $\rows$ and $\sigma$ for every $d$-subset $\sigma\subset\ground$.
\end{definition}

We describe another (larger) matching field induced from a triangulation, which comprises the full information of the original triangulation.
We augment the ground set $\ground$ by a copy $\widetilde{\rows}$ of $\rows$ to obtain a ground set $\augground$ of size $n+d$, and we set all elements of $\widetilde{\rows}$ to be smaller than all elements of $\ground$.
%\todo{YCH: Do we really need the ordering of $\augground$ in this paper?}
The collection $\widetilde{\mathcal{T}}$ contains, for every tree $T$ in $\mathcal{T}$, the tree on $\rows \sqcup \augground$ obtained from $T$ by adding an edge between $i$ and its copy for every $i\in\rows$.

%A special class of polyhedral matching fields are the \emph{pointed} ones (following the terminology in~\cite{SturmfelsZelevinsky:1993}), which comprise the full information of a triangulation of $\ssimplex_{d-1}\times\ssimplex_{n-1}$.
%We augment the ground set $\ground$ by a copy $\widetilde{\rows}$ of $\rows$ to obtain a ground set $\augground$ of size $n+d$, and we set all elements of $\widetilde{\rows}$ to be smaller than all elements of $\ground$.
%To take the full information of all trees into account, for each tree in $\mathcal{T}$ we add edges between each node in $\rows$ to its copy $\widetilde{\rows}$, yielding a set $\widetilde{\mathcal{T}}$ of trees on $\rows \sqcup \augground$.

\begin{definition} \label{def:pointed+polyhedral}
  The \emph{pointed polyhedral matching field} associated with a triangulation $\mathcal{T}$ of $\ssimplex_{d-1}\times\ssimplex_{n-1}$ is the collection of $\rows$-saturating matchings on $\rows \sqcup \augground$ that are subgraphs of the trees in $\widetilde{\mathcal{T}}$. 
\end{definition}

%The discussion after \cite[Theorem 3.16]{LohoSmith:2020} shows that the latter construction actually yields a correspondence between triangulations and matching fields.

%One can see that pointed polyhedral matching fields are actually polyhedral matching fields by extending the original triangulation of $\ssimplex_{d-1}\times\ssimplex_{n-1}$ to a triangulation of $\ssimplex_{d-1}\times\ssimplex_{n+d-1}$ by using a placing triangulation as in~\cite[Lemma~4.3.2]{DeLoeraRambauSantos:2010}.
%Note that, on the other hand, each polyhedral matching field is a sub-matching field of a pointed polyhedral matching field.
%We use both points of view as they allow different constructions as we see in Section~\ref{sec:patchwork}.
%It is similar to the relation between transversal matroids and fundamental transversal matroids, and it is reminiscent of the correspondence between matroids and linking systems~\cite{Schrijver:1979}.

\bigskip

%Cutting appropriately through a triangulation of $\ssimplex_{d-1}\times\ssimplex_{n-1}$ yields a \emph{fine mixed subdivision} of $n\ssimplex_{d-1}$.
%This is formalized as the \emph{Cayley trick}, see~\cite{Santos:2005}. 
The \emph{Cayley trick} \cite{Santos:2005} establishes a bijective correspondence between triangulations of $\ssimplex_{d-1}\times\ssimplex_{n-1}$ and \emph{fine mixed subdivisions} of the dilated simplex $n\ssimplex_{d-1}$ as follows:
%A direct way to identify the polyhedral pieces in such a subdivision of $n\ssimplex_{d-1}$ is the following.
For each tree $G$ corresponding to a simplex in the triangulation $\mathcal{T}$, we form the Minkowski sum
\[
\sum_{j \in \ground} \conv \{ {\bf e}_i \colon i \in \mathcal{N}_G(j) \} \enspace ,
\]
where $\neighbourhood_G(j)$ is the neighbourhood of an element $j \in \ground$ in $G$.
The collection of these Minkowski sums tiles $n\ssimplex_{d-1}$.

In \cite{ArdilaDevelin:2009}, Ardila and Develin studied the dual of these mixed subdivisions as \emph{tropical pseudohyperplane arrangements}, which generalizes tropical hyperplane arrangements as the dual of coherent mixed subdivisions~\cite{DevelinSturmfels:2004}.
In \cite{Horn1,OhYoo:2011}, it was shown that these objects are equivalent to {\em tropical oriented matroids}, defined by purely combinatorial axioms back in \cite{ArdilaDevelin:2009}.
%The dual polyhedral complexes to these mixed subdivisions were introduced as \emph{tropical pseudohyperplane arrangements} in~\cite{ArdilaDevelin:2009}.
%This draws from the correspondence between tropical hyperplane arrangements and regular subdivisions of products of simplices established in~\cite{DevelinSturmfels:2004}.
%After starting from an axiomatic study of \emph{tropical oriented matroids} in~\cite{ArdilaDevelin:2009}, it was subsequently shown that these combinatorial objects are indeed \emph{cryptomorphic} to subdivisions of $\ssimplex_{d-1}\times\ssimplex_{n-1}$ and tropical pseudohyperplane arrangements, partially in~\cite{OhYoo:2011} and finished in~\cite{Horn1}.

\begin{figure}[htb]
  \centering
  \includegraphics[scale=0.5]{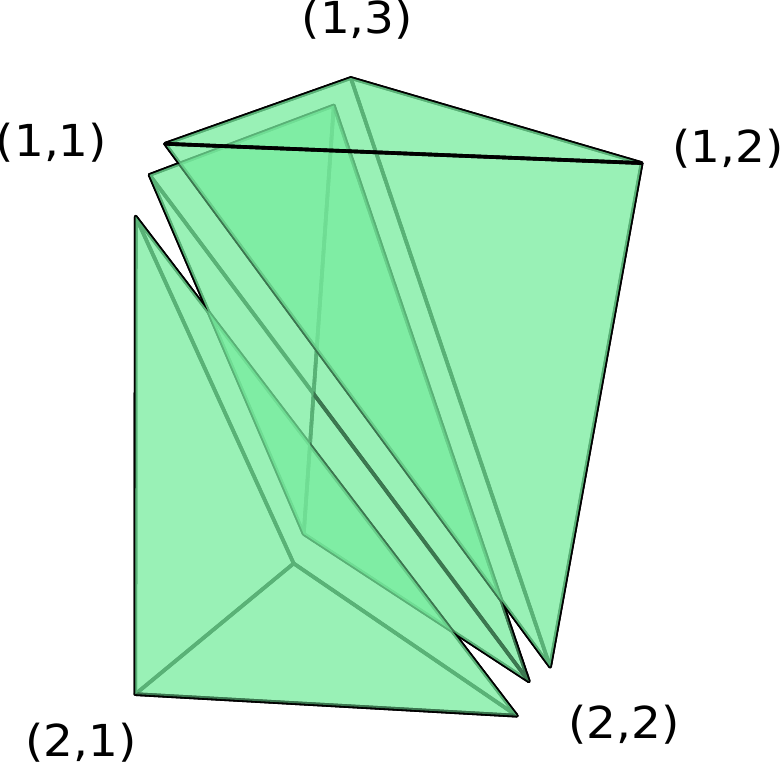}
  \caption{A triangulation of $\ssimplex_1\times\ssimplex_2$. The vertices are labeled by the corresponding edges in $K_{2,3}$. This picture was created with \texttt{polymake} \cite{DMV:polymake}.}
  \label{fig:regular-subdivision-prism}
\end{figure}

\subsection{Oriented Matroids from Triangulations of $\ssimplex_{d-1}\times\ssimplex_{n-1}$}\label{sec:part+I}

%\todo{GL: Do we want to have the corresponding numbers of the results in Part I here? }

We recall the results of \cite{CelayaLohoYuen:2020} which are needed in this paper.
For the rest of this paper, we fix a sign matrix $A\in\{-,+\}^{\rows\times\ground}$ and a polyhedral matching field $(M_{\sigma})$ extracted from a triangulation of $\ssimplex_{d-1}\times\ssimplex_{n-1}$.
We also denote by $\widetilde{(M_{\sigma})}$ the pointed polyhedral matching field encoding the starting triangulation, and by $\widetilde{A} $ the sign matrix $(I_{d,d} | A)$. 

Using the ordering on $\rows$ and $\sigma\subset\ground$, we can interpret a matching $M_{\sigma}$ as a permutation, and we define the sign of the matching by the sign of the permutation.

\begin{theorem}\cite[Theorem~A]{CelayaLohoYuen:2020} \label{thm:main}
The sign map $\rchi:\binom{\ground}{d}\rightarrow\{+,-\}$ given by
\begin{equation} \label{eq:sign+map+matching+field}
  \sigma\mapsto\sign(M_\sigma)\prod_{e\in M_\sigma}A_e \enspace ,
\end{equation}
is the chirotope of an oriented matroid.

%Moreover, for each tree $T\in\Ical$, the restriction of $\rchi$ to the bases of $M(T)$ is a realizable chirotope realized by any matrix  $A(T)$ obtained in the following manner: set every entry of $A$ not in $T$ to $0$ and replace every remaining non-zero entry by an arbitrary real number of the same sign.
\end{theorem}

We denote the oriented matroid described by $\rchi$ as  $\orientedmatroid$.
Similarly, $\widetilde{(M_{\sigma})}$ induces an oriented matroid $\widetilde{\orientedmatroid}$ on $\augground$.\\

Now, we describe how to convert cells of a fine mixed subdivision of  $n\ssimplex_{d-1}$, as special subgraphs of $K_{\rows,\ground}$, into a covector of $\orientedmatroid$ (resp. $\widetilde{\orientedmatroid}$).

\begin{definition}\cite[Definition~3.26]{CelayaLohoYuen:2020} \label{def:signed+forest}
Given $S\in\{-1,0,1\}^{\rows}$ and $F\subseteq\rows\times\ground$, the sign matrix $SA_{F}\in\left\{ -1,0,1\right\} ^{\rows\times\ground}$ is defined as
\[
(SA_{F})_{i,j}=\begin{cases}
S_{i}A_{i,j}, & (i,j)\in F,\\
0, & \text{otherwise.}
\end{cases}
\]
\end{definition}

\begin{definition}\cite[Definition~3.27]{CelayaLohoYuen:2020} \label{def:triang+to+covectors+OM}
  Given a subgraph $F$ of $K_{\rows,\ground}$
  and a sign vector $S\in\{-1,0,1\}^{\rows}$, the sign vector $\psi_A(S,F)= Z \in\{-1,0,1\}^{\ground}$ is given by 
\[
Z_{j}=
\begin{cases}
0, & \text{column \ensuremath{j} of \ensuremath{SA_{F}} contains positive and negative entries, or all zeros}\\
1, & \text{column \ensuremath{j} of \ensuremath{SA_{F}} contains only non-negative entries}\\
-1, & \text{column \ensuremath{j} of \ensuremath{SA_{F}} contains only non-positive entries.}
\end{cases}
\]
\end{definition}

\begin{proposition}\cite[Proposition~3.32]{CelayaLohoYuen:2020} \label{prop:covectors+from+topes}
Let $F$ be a subgraph of a tree in $\widetilde{\mathcal{T}}$ without isolated nodes in $\ground \subset \widetilde{\ground}$, and such that a node in $\widetilde{\rows} \subset \widetilde{\ground}$ is isolated only if the corresponding node in $\rows$ is isolated as well. 
Let $S \in \{-1,0,1\}^{\rows}$ be a sign vector whose support contains the set of non-isolated nodes of $F$ in $\rows$.

Then the sign vector $\psi_{\widetilde{A}}(S,F)$ is a covector of $\widetilde{\orientedmatroid}$. 
\end{proposition}

%Restricting the covectors to $\ground$ yields the following. 

%\todo{GL: We might have to check with covector graph vs. pd-graph here. }

\begin{corollary}\cite[Corollary~3.33]{CelayaLohoYuen:2020} \label{coro:more+general+covectors}
  For a subgraph $F$ of a tree in $\mathcal{T}$ with no isolated node in $\ground$, the sign vector $\psi_A(S,F)$ is a covector of $\orientedmatroid$ for every sign vector $S$.
\end{corollary}

Note that the latter subgraphs are called \emph{covector pd-graphs} in~\cite{CelayaLohoYuen:2020}. 

%We shall see in Section~\ref{sec:patchwork} that the map is actually surjective.

\subsection{Poset and Lattice Quotients}
\label{sec:poset_lattice_quotients}
The following definition is due to Hallam and Sagan \cite{HallamSagan:2015},
which proved useful in their work on factorizing characteristic polynomials
of lattices. This definition turns out to be the right one for us
as well.

\begin{definition}

Let $\poset$ be a finite poset. An equivalence relation $\sim$ on
the ground set of $\poset$ is $\poset$-\emph{homogeneous} provided
the following condition holds: if $\tau\leq\sigma$ in $\poset$,
then for every $u\in\eqclass{\tau}$ there exists $v\in\eqclass{\sigma}$
such that $u\leq v$ in $\poset$. We denote by either $\eqclass{\sigma}$
or $\sigma\eqrel$ the equivalence class of $\sigma$ in $\sim$. 

\end{definition}

\begin{proposition}[{\cite[Lemma~5]{HallamSagan:2015}}] Suppose $\sim$
is $\poset$-homogeneous. Then we have a well-defined poset $\poset\eqrel$
on the classes of $\sim$ defined as follows: $\eqclass{\tau}\leq\eqclass{\sigma}$
in $\poset\eqrel$ if and only if there exists $u\in\eqclass{\tau}$
and $v\in\eqclass{\sigma}$ such that $u\leq v$ in $\poset$. Equivalently,
for every $u\in\eqclass{\tau}$ there exists $v\in\eqclass{\sigma}$
such that $u\leq v$ in $\poset$.

\end{proposition}

We call the poset $\poset\eqrel$ the \emph{homogeneous quotient of}
$\poset$ \emph{by} $\sim$. We are particularly interested in homogeneous
quotients which have nice factorizations in the following sense:

\begin{definition}A homogeneous quotient $\poset\eqrel$ is an \emph{elementary
quotient} if every equivalence class of $\sim$ is either a singleton, or consists of exactly three elements $\sigma,\tau,\gamma\in\poset$
such that $\sigma$ and $\tau$ both cover $\gamma$ in $\poset$.

\end{definition}

\begin{definition}We say that $\poset\eqrel$ \emph{admits a factorization
into elementary quotients} if there exist posets $\poset=\poset_{0},\poset_{1},\ldots,\poset_{k}=\poset\eqrel$
such that $\poset_{i}=\poset_{i-1}\eqrel_{i}$ is an elementary quotient
of $\poset_{i-1}$ for all $i=1,2,\ldots,k$.

\end{definition}

Since the following notion appears several times in this paper, we
give the definition here:

\begin{definition}The \emph{augmented poset} of a poset $\poset$
is the poset $\lattice(\poset):=\poset\cup\{\hat{\zero},\hat{\oneone}\}$,
where $\hat{\zero}$ and $\hat{\oneone}$ are two additional elements
such that $\hat{\zero}<\sigma<\hat{\oneone}$ for all $\sigma\in\poset$.

\end{definition}

\section{Patchworking Pseudosphere Arrangements}\label{sec:patchwork}

\subsection{Patchworking Pseudolines on an Example}

The classical theory of patchworking states that the structure of the real zero set of a polynomial in one orthant, parameterized by $t > 0$, is captured for sufficiently small $t$ by the regular triangulation of its Newton polytope induced by the exponents of $t$.
Hence, one can recover the structure of the real zero set by gluing the triangulations for all orthants. 
This uses an appropriate assignment of signs to the vertices of the Newton polytope.
By considering coherent fine mixed subdivisions, see Section~\ref{sec:triangulation}, this was extended to complete intersections in~\cite{Sturmfels:1994a}.

We use patchworking of not-necessarily coherent fine mixed subdivisions of $n\ssimplex_{d-1}$ to derive a representation theorem for the oriented matroids induced from polyhedral matching fields.
This can be seen as a generalization of the linear case of~\cite[Thm.~4]{Sturmfels:1994a} for generic hyperplane arrangements.
While a complete intersection for generic hyperplanes would only yield one specific cell, the oriented matroid captures the information of all intersections in a generic hyperplane arrangement. 

Example~\ref{ex:small_patchworking} is a toy example that illustrates our construction, which is generalized to larger $\ground$ and higher rank in this section. 

\begin{figure}[htb]
\includegraphics[width=.7\textwidth]{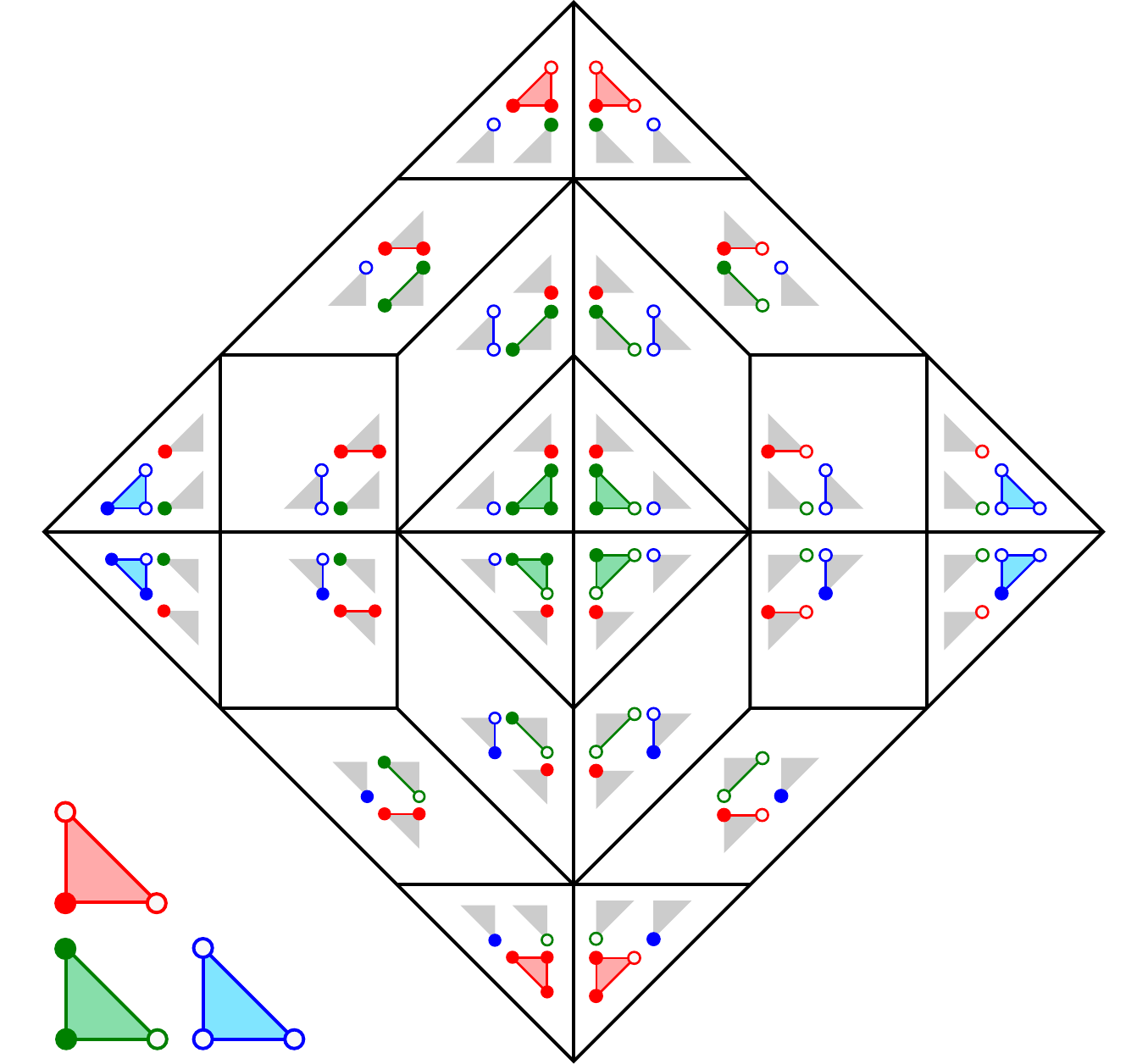}

  \caption{Fine mixed subdivision of $3\ssimplex_2$ with cells labeled by their summands and sign; filled vertices denote `$+$', empty ones~`$-$'.}
  \label{fig:mixed+subdivision+labeled}
\end{figure}

\begin{figure}[htb]
    
\includegraphics[width=.5\textwidth]{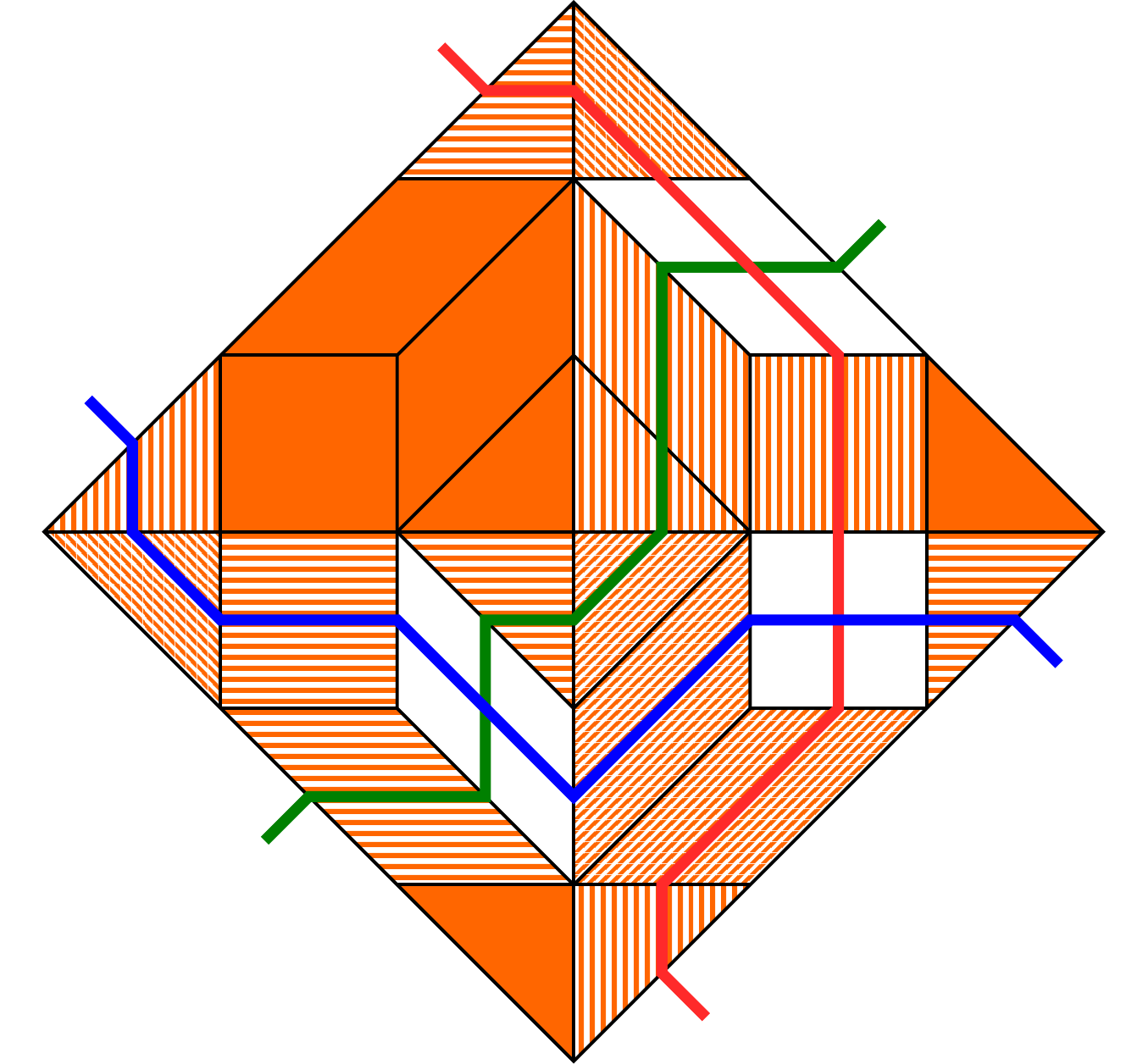}
  \caption{Pseudohyperplane arrangement derived from a fine mixed subdivisions of $3\ssimplex_2$ with signs indicated in Figure~\ref{fig:mixed+subdivision+labeled}. }
  \label{fig:mixed+subdivision+PL+curves}
\end{figure}
\begin{example} \label{ex:small_patchworking}
We start with the regular triangulation of $\ssimplex_2\times\ssimplex_2$ induced by the height matrix
\[
H = 
  \begin{pmatrix}
    0 & 3 & 2 \\
    0 & 0 & 0 \\
    1 & 3 & 0 
  \end{pmatrix} \; .
  \]
  This gives rise to the following height function on the lattice points of $3 \ssimplex_2$:
  \[
  (300; 5), (201; 6), (210; 5), (102; 6), (111; 5), (120; 3), (003; 4), (012; 4), (021; 3), (030; 0)
  \]
  Here, the height of the lattice point $(p_1,p_2,p_3)$ is the weight of the maximal matching on $K_{3,3}$ for which the weight function is obtained from $H$ by taking $p_{\ell}$ copies of the $\ell$-th row of $H$. 
  Note that, alternatively, the latter height function of the mixed subdivision can be obtained by multiplying the $\max$-tropical linear polynomials $(x_0 \oplus 3 \odot x_1 \oplus 3 \odot x_2) \odot (x_0 \oplus x_1 \oplus 2 \odot x_2) \odot (x_0 \oplus 1 \odot x_1 \oplus x_2)$. 
  We refer the reader further interested in this connection to~\cite{Joswig:2020}. 
  
Additionally, we equip the subdivision by the sign matrix
    \[
  \begin{pmatrix}
    - & + & - \\
    + & + & - \\
    - & - & - 
  \end{pmatrix} \enspace .
  \]

The fine mixed subdivision of $3\ssimplex_2$ induced by $H$ is shown in the upper-right quartile of Figure~\ref{fig:mixed+subdivision+labeled}.
The cells are labeled by their Minkowski summands (cf. the Cayley trick in Section~\ref{sec:triangulation}) as follows.
The elements of $\ground$ (as the three columns from left to right) are represented by the red, green, and blue simplices, respectively; the elements of $R$ (as the three rows from top to bottom) are represented by the top, lower-left, and right vertices of each simplex, respectively.
Furthermore, the vertices of these simplices are labeled with signs coming from the sign matrix.

The faces of a mixed cell correspond to the subgraphs of its spanning tree without isolated nodes in $\ground$. 
In particular, a vertex of a mixed cell can be specified by choosing a vertex from each colored simplex, thus it encodes a sign vector $\{+,-\}^\ground$.
Such a forest associated with a vertex is independent of the mixed cell containing it.
Hence, we have a well-defined assignment of sign vectors to the vertices of the mixed subdivision.
For example, the lower left vertex $v$ of the square in the upper-right quartile of Figure~\ref{fig:mixed+subdivision+labeled} is the Minkowski sum of a filled red vertex, an empty green and an empty blue vertex.
Therefore, it encodes the sign vector $(+,-,-)$. 

Next we reflect the dilated simplex across the coordinate hyperplanes in $\mathbb{R}^3$ so that there is a copy in every octant\footnote{We only show the upper half of that {\em patchworking complex} in Figure~\ref{fig:mixed+subdivision+labeled} as the construction is centrally symmetric.}.
We keep the same subdivision in all copies and label the vertices of these copies with sign vectors similar to the above, but instead of the original sign matrix, we negate a row of it if the corresponding coordinate in the octant is negative.
For the vertex $v$, e. g., this yields $(+,+,-)$ for its reflection in the upper-left quartile of Figure~\ref{fig:mixed+subdivision+labeled}. 
Again, the sign vector assigned to a vertex that appears in multiple copies of the dilated simplex is independent of the copy chosen: whenever a vertex lies on a hyperplane $\{x_i=0\}$, the $i$-th node of $R$ must be an isolated one in the forest corresponding to the vertex, thus the negation of the $i$-th row does not affect the sign vector.

As our example is of rank 3, we obtain a subdivision of the boundary of a dilated octahedron (which is PL homeomorphic to $S^2$), with vertices of the subdivision labeled by sign vectors.

Finally, we define a ``zero locus'' for each element $e\in\ground$ as a subset of the patchworking complex.
This zero locus is dual to the cells which have a Minkowski summand with vertices of different sign. 
Given a cell of the subdivision, select the edges (one-dimensional faces) of the cell in which the sign vectors of their endpoints disagree on the $e$-th coordinate, and take the convex hull of the midpoints of them.
Take the union of all such convex hulls, it can be seen from Figure~\ref{fig:mixed+subdivision+PL+curves} that each of such ``zero loci'' is a pseudosphere on the patchworking complex.

Note that the boundary of $3\ssimplex_2$ in $\RR^3$ can be seen as the intersection with the three hyperplanes bounding the non-negative orthant.
Extending these through the reflections of $3\ssimplex_2$ yields three further pseudospheres.
This gives rise to an interpretation of Figure~\ref{fig:mixed+subdivision+PL+curves} as an arrangement of six pseudospheres.
By `fattening' the latter three coordinate pseudospheres we arrive at the \emph{extended patchworking complex} introduced in the next section. 

\end{example}

\begin{remark} \label{rem:Ringel}
Since oriented matroids coming from regular triangulations are all realizable, the pseudosphere arrangements constructed from a coherent fine mixed subdivision are all stretchable, which gives some non-trivial structural constraints on coherence.

We recall the example from~\cite{CelayaLohoYuen:2020} that treats Ringel's non-realizable uniform oriented matroid $\mathbf{R}$ of rank $3$ on $9$ elements.
It can be realized by patchworking a suitable non-coherent fine mixed subdivision of $6 \ssimplex_2$ and choosing appropriate signs as depicted in Figure~\ref{fig:realizing+ringel}.
  It is not clear to the authors, if the oriented matroid $\mathbf{R}$ can be constructed from another non-coherent fine mixed subdivision of $6 \ssimplex_2$.

It is an interesting experimental question of which of the 24 (non-isomorphic) non-realizable oriented matroids of rank $4$ on $8$ elements arise from the non-regular triangulation constructed by de Loera in~\cite{DeLoera:1996} or its modifications by choosing appropriate signs.
\end{remark}

\begin{figure}
    \centering
    \includegraphics[scale=0.35]{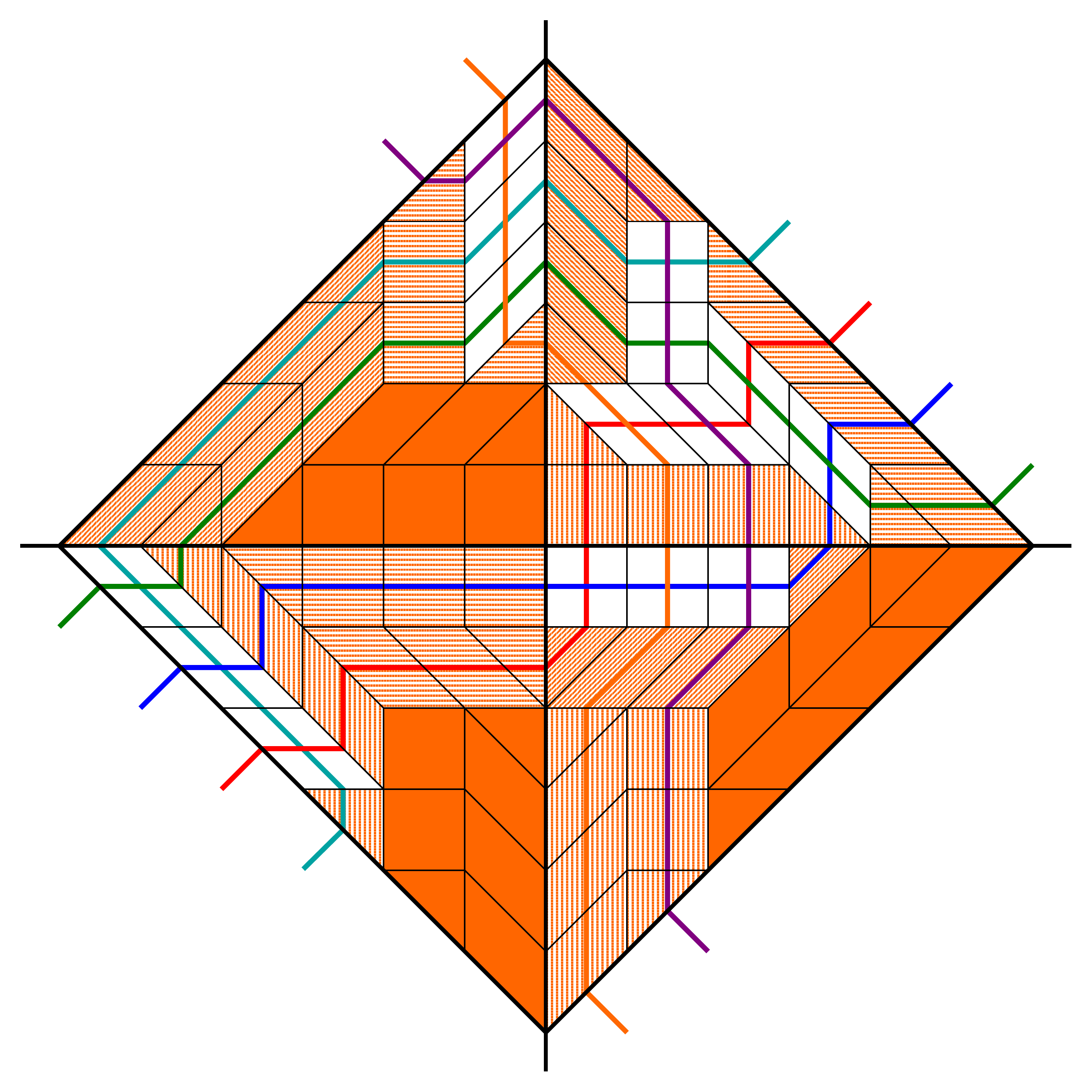}
    \caption{Non-coherent fine mixed subdivision of $6\Delta_2$ patchworking the Ringel arrangement discussed in Remark~\ref{rem:Ringel}. }
  \label{fig:realizing+ringel}
\end{figure}

\begin{remark}
It is also an interesting problem to interpret our construction here as a limit with respect to some one-dimensional family of (meaningful) geometric objects; the regular (and non-singular) case is closely related to the theory of amoebas \cite{Mikhalkin:2004}.
The rank 3 case is worth to put emphasis on, not only because it is already combinatorially rich enough, but the work of Ruberman--Starkton indicates that every pseudoline arrangement can be complexified into an arrangement of symplectic spheres~\cite{RubermanStarkton:2019}, hence suggesting a symplectic flavored answer here (see also the aforementioned work of Itenberg--Shustin \cite{ItenbergShustin:2002}).
\end{remark}

\subsection{From Fine Mixed Subdivisions to Pseudosphere Arrangements}

\label{sec:mixed_subdiv_to_PSA}

We now state precisely our method for constructing a pseudosphere
arrangement representing an oriented matroid associated to a polyhedral
matching field. For this, we fix a fine mixed subdivision~$\fmixed$
of~$n\ssimplex_{d-1}$. By the Cayley trick, this corresponds to
a triangulation $\mathcal{T}$ of $\ssimplex_{d-1}\times\ssimplex_{n-1}$.
By means of Definition~\ref{def:pointed+polyhedral}, it gives
rise to a pointed polyhedral matching field $(\widetilde{M_{\sigma}})$
on $\rows\sqcup\augground$ with $\augground=\widetilde{\rows}\cup\ground$. For
an arbitrary matrix $A\in\{+,-\}^{\rows\times\ground}$, we consider
the augmented matrix $\widetilde{A}=(I_{\rows}\mid A)$ as sign matrix
for $(\widetilde{M_{\sigma}})$. Let $\widetilde{\orientedmatroid}$
denote the oriented matroid on $\widetilde{\ground}=\widetilde{\rows}\cup\ground$
associated to the pointed polyhedral matching field $(\widetilde{M_{\sigma}})$
with the sign matrix $\widetilde{A}$, and let $\orientedmatroid$
be its restriction to $\ground$.

Recall from Section~\ref{sec:triangulation} that we may identify
the maximal simplices in $\mathcal{T}$ with spanning trees of $K_{\rows,\ground}$.
The cells $\sigma_{F}$ of $\fmixed$ are in 1-1 correspondence with
the forests $F$ contained in a spanning tree of $\mathcal{T}$ for
which $\deg_{F}(j)\geq1$ for all $j\in\ground$.

\smallskip{}
We denote the cube $[-1,1]^{d}\subset\RR^{d}$ and its polar dual,
the crosspolytope, by $\cube_{d}$ and $\crosspolytope_{d}$, respectively.
For a sign vector $S\in\left\{ -1,0,1\right\} ^{d}$ and a set $K$
contained in the coordinate subspace $\RR^{\supp(S)}\times\{0\}^{\overline{\supp(S)}}$ of $\RR^{d}$,
define
\begin{align*}
S\cdot K & :=\left\{ (S_{1}x_{1},\ldots,S_{d}x_{d})\in\RR^{d}:(x_{1},\ldots,x_{d})\in K\right\} \\
\cube_{S} & :=\left\{ x\in\cube_{d}:x_{i}=S_{i}\text{ for all \ensuremath{i\in\supp(S)}}\right\} .
\end{align*}
Hence, $S\cdot K$ denotes the reflections of $K$ to the orthant
indicated by $S$, and $\cube_{S}$ comprises the sign patterns of
orthants containing the sign vector $S$. For a subgraph $F$ of $K_{\rows,\ground}$,
let $\supp_{\rows}(F):=\left\{ i\in\rows:\deg_{F}(i)\geq1\right\} $.
This set encodes the unique minimal face of $n\ssimplex_{d-1}$ containing
$F$. 

\begin{proposition}\label{prop:S_octagon}

The subdivision $\fmixed$ of $n\ssimplex_{d-1}$ gives rise to the
subdivision
\[
\octfmixed:=\left\{ \sigma_{(S,F)}:\sigma_{F}\in\fmixed,\;S\in\{-1,0,1\}^{\ground},\;\supp(S)\supseteq\supp_{\rows}(F)\right\} 
\]
of the boundary of $\octagon_{d}:=\cube_{d}+n\crosspolytope_{d}$,
where $\sigma_{(S,F)}:=\cube_{S}+S\cdot\sigma_{F}$.

\end{proposition}

We call the complex arising in the latter Proposition the \emph{extended patchworking complex}; we prove the statement together with more technical properties of the extended patchworking complex in Section~\ref{sec:properties+extended+patchworking+complex}. This complex is analogous to the complex $\Delta_{\omega}'$ defined in \cite[Theorem 5]{Sturmfels:1994a}, with additional cells that are dual to the  coordinate hyperplanes of $\RR^d$. The extended patchworking complex subdivides the boundary of a polytope,
and is therefore a PL sphere. Hence, we may consider its dual complex
\[
\Delta:=\octfmixed^{\dual}:=\{\sigma_{(S,F)}^{\dual}:\sigma_{(S,F)}\in\octfmixed\}.
\]
The realization of the poset as a polyhedral cell complex is further
explained in Section~\ref{sec:basics+cell+complexes}. For $i\in\widetilde{\rows}$
and $j\in\ground$, define the subcomplexes
\begin{align}
\Delta_{i} & :=\{\sigma_{(S,F)}^{\dual}\in\Delta:i\notin\supp(S)\},\label{eq:Delta_i}\\
\Delta_{j} & :=\{\sigma_{(S,F)}^{\dual}\in\Delta:\text{there exist edges \ensuremath{(i,j),(\ell,j)} in \ensuremath{F}}\label{eq:Delta_j}\\
 & \phantom{:==\{\sigma_{(S,F)}^{\dual}\in\Delta:}\text{such that }S_{i}A_{i,j}=-S_{\ell}A_{\ell,j}\neq0\}.\nonumber 
\end{align}

Recall the notion of a pseudosphere arrangement from Definition~\ref{def:pseudosphere+arrangement}.

\begin{theorem}\label{thm:arrangement_of_pseudospheres_extended}

The spaces $\norm{\Delta_{k}}$ ranging over all $k\in\widetilde{\ground}$
form an arrangement of pseudospheres within $\norm{\Delta}$ representing
the oriented matroid $\widetilde{\orientedmatroid}$.

\end{theorem}

Deleting the pseudospheres $\norm{\Delta_{i}}$ for $i\in\widetilde{\rows}$
yields the following.

\begin{corollary} The spaces $\norm{\Delta_{j}}$ ranging over all
$j\in\ground$ form an arrangement of pseudospheres representing the
oriented matroid $\orientedmatroid$. \end{corollary}

\begin{figure}
\begin{centering}
\includegraphics{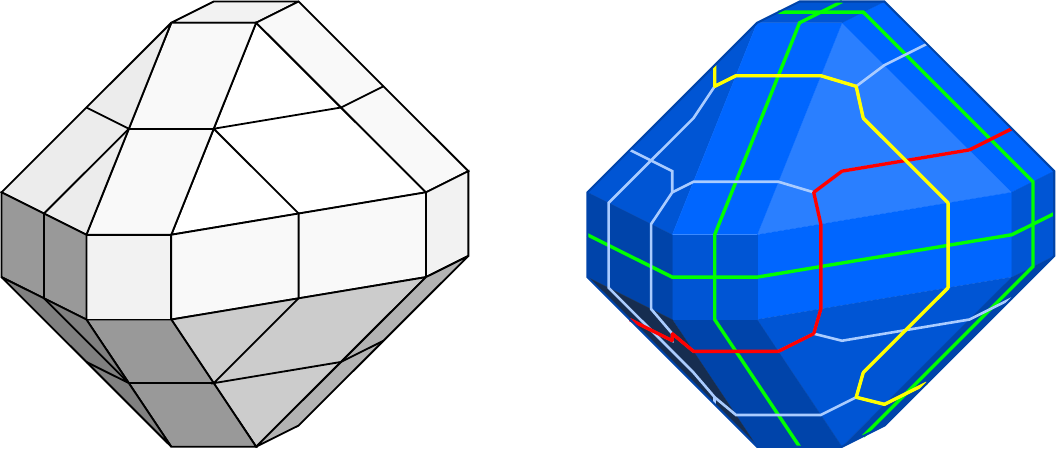}
\par\end{centering}
\caption{The complex $\protect\octfmixed$ (left) and its dual $\Delta=\protect\octfmixed^{\protect\dual}$
(right). The five subcomplexes of $\Delta$ that yield pseudospheres
are highlighted: there are three of the form $\Delta_{i},\;i\in\protect\rows$
(shown in green) and two of the form $\Delta_{j},\;j\in\protect\ground$
(shown in red and yellow).}
\end{figure}

\subsection{Overview of the Proof of Theorem~\ref{thm:arrangement_of_pseudospheres_extended}}

\label{subsec:proof_of_patchworking_theorem}
%Since it involves a lot of ingredients, we spread the proof of Theorem \ref{thm:arrangement_of_pseudospheres_extended} out over this subsection, with some of the more technical proofs relegated to Section~\ref{sec:pf_main_theorem}.

Before getting into the technical details of our proof, we explain the overall picture.
The codimension one skeleton of $\Delta$ in each orthant is a tropical pseudohyperplane arrangement in the sense of~\cite{ArdilaDevelin:2009}, that is, a union of PL homeomorphic images of tropical hyperplanes (codimension one skeleton of $\ssimplex_{d-1}^{\dual}$).
Including the sign data, each $\Delta_k$ restricted to an orthant is either empty or is (the boundary of) a {\em tropical (pseudo)halfspace} in the sense of \cite{Joswig:2005}.
The latter is obtained from a tropical pseudohyperplane by removing the facets that lie between two regions of the same sign.
As such, the arrangement of $\Delta_k$'s can be thought as the end product of a facet removal process for multiple tropical pseudohyperplanes across multiple orthants.

Using the results from \cite{CelayaLohoYuen:2020} (as summarized in Section~\ref{sec:part+I} here), we can show that the face poset of the arrangement of $\Delta_k$'s equals the covector lattice of $\orientedmatroid$.
The challenge now becomes topological: we need to make sure that the facet removal process does not create pathologies, so the topological structure reflects the combinatorial structure.
Our approach is to formulate this process as a stepwise cell merging process, using the formalism of regular cell complexes.
The removal of a facet determines an equivalence relation on the cells of the tropical pseudohyperplane arrangement: two cells are equivalent if their interiors intersect the interior of a common cell once the facet is removed.
By taking the union of the cells in each equivalence class, we show that we get another regular cell complex with the same underlying
topological space.
Iterating this procedure, we show that we end up at a regular cell complex.
Since the face poset of a regular cell complex determines the complex up to cellular homeomorphism, this completes the proof. Figure~\ref{fig:cell_merging_example} depicts a two-dimensional example where na\"ive cell merging does not preserve regularity, hence justisfying the technical work here.

Hersh describes a similar step-by-step process in \cite[\S 4]{Hersh:2014} to simplify a regular cell complex while preserving homeomorphism type.
Her single step involves collapsing a single cell to a cell on the boundary, while ours involves merging two neighbouring cells together.
Thus, in some sense, her approach might be considered dual to ours.

We remark that the work in this section is very closely related to the work done by Horn in \cite[Chapter 6]{Horn1}.
Indeed, one approach to proving Theorem~\ref{thm:arrangement_of_pseudospheres_extended} is to directly use her second Topological Representation Theorem for tropical oriented matroids. Taking this approach, one would then show that the $2^d$ affine pseudohyperplane arrangements guaranteed by Horn's theorem (some can be empty) glue together to form a pseudosphere arrangement, and that this pseudosphere arrangement represents the desired oriented matroid $\orientedmatroid$.

However, one of our goals in this paper is to furnish a proof of Theorem~\ref{thm:arrangement_of_pseudospheres_extended} that is almost entirely combinatorial. We achieve this by making use of the correspondence between tropical hyperplane arrangements and generic tropical oriented matroids, as well as the correspondence between a regular cell complex and its face poset. In particular, we show how the \emph{elimination axiom} of tropical oriented matroids enables our cell merging process to work, which might lead to extensions of our method. 

\begin{figure}
\begin{centering}
\includegraphics{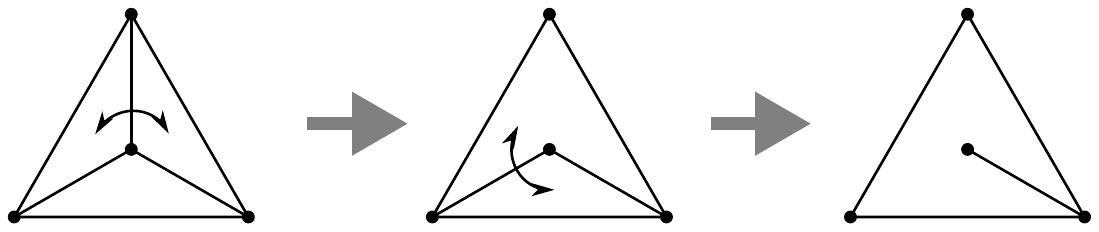}
\par\end{centering}
\caption{Two cell-merging steps of a planar embedding of the complete graph $K_4$ into the plane. The first merging step results in a regular CW complex, however the second does not.}
\label{fig:cell_merging_example}
\end{figure}

\subsection{Relation to Real Bergman Fan and Complex}

To close this section, we sketch the relation of our construction with the {\em real Bergman fan} of the oriented matroid, considered by the first author in \cite[Chapter~2]{Celaya:2019}.
The real Bergman fan generalizes to oriented matroids the more well-known {\em Bergman fan} of a matroid, which is itself the tropical analogue of a linear subspace.
In particular, the Bergman fan of a matroid is the union of cones taken over all flags of flats, whereas the real Bergman fan is the union of cones over all flags of {\em conformal covectors}:

\begin{definition}
Let $\Lcal$ be the collection of nonzero covectors of an oriented matroid $\orientedmatroid$. Identify each sign vector $X\in\Lcal$ with its associated lattice point ${\bf e}_X\in [-1,1]^E\cap\ZZ^E$.
Then the {\em real Bergman fan} $\Sigma_\orientedmatroid^*$ of $\orientedmatroid$ is the collection of all cones of the form
$$
\cone\{{\bf e}_{X_1},\ldots,{\bf e}_{X_k}\}\subset\RR^E
$$
where $X_1<\ldots<X_k$ and each $X_i\in\Lcal$. The {\em real Bergman complex} $\Delta_\orientedmatroid^*$ of $\orientedmatroid$ is the intersection of $\Sigma_\orientedmatroid^*$ with the boundary of the hypercube $[-1,1]^E$.
\end{definition}

We note that, after taking the componentwise logarithm, the real Bergman fan (resp. complex) restricted positive orthant coincides with the positive Bergman fan (resp. complex) considered by Ardila, Klivans, and Williams \cite{AKW:2006}. Conversely, the real Bergman fan can be recovered from the positive Bergman fans of all reorientations of $\orientedmatroid$. Such fans were used in the realizable setting by J\"urgens in \cite{Jurgens:2018}. See \cite[Chapter~2.4]{Celaya:2019} for further combinatorial properties of the real Bergman fan.

The complex $\Delta_{\orientedmatroid}^*$ is a geometric realization of the order complex of $\Lcal$. Hence, a direct consequence of the Topological Representation Theorem is that this complex is PL homeomorphic to a sphere of dimension $d-1$, and its intersection with the coordinate hyperplanes of $\RR^E$ are the pseudospheres representing the elements. It is therefore natural to ask if there is a piecewise linear map from the extended patchworking complex $\Delta=\octfmixed^{\dual}$ defined in Section~\ref{sec:mixed_subdiv_to_PSA} to the real Bergman complex of the associated oriented matroid $\widetilde{\orientedmatroid}$, one which respects the pseudosphere arrangement structure. This can indeed be carried out; we omit the details as they are routine:

\begin{proposition}Define the following map on the vertices $\sigma_{(S,F)}$ of $\octfmixed$ into $\RR^{\widetilde{\ground}}$:
\[
\sigma_{(S,F)}\mapsto\left( S, \oneone^\top (SA_F) \right)\in \RR^{\widetilde{\ground}}.
\]
Here $SA_F$ is the matrix as in Definition~\ref{def:signed+forest}, and $\oneone$ denotes the vector of all ones. Extend this map linearly on each maximal cell of $\octfmixed$, to get a map 
\[
\norm{\Delta}=\bigl\Vert\octfmixed\bigr\Vert\rightarrow\RR^{\widetilde{\ground}}. 
\]
Then this map is well defined, and the image of this map is precisely $\Delta^*_{\widetilde{\orientedmatroid}}$.
 
Furthermore, the choices implicit in the construction of $\octfmixed^{\dual}$ can be made so that this map respects the cellular structure of the pseudosphere arrangement $\norm{\Delta_{k}}$ over all $k\in\widetilde{\ground}$ as given by Theorem~\ref{thm:arrangement_of_pseudospheres_extended}, and the pseudosphere arrangement obtained by intersecting $\Delta^*_{\widetilde{\orientedmatroid}}$ with each of coordinate hyperplane of  $\RR^{\widetilde{\ground}}$.
\end{proposition}
%\footnote{If we work with the pointed setting, i.e., appending an identity matrix to the image matrices, then we can naturally embed the patchworking complex into a Grassmannian.}

\begin{example}
In Figure~\ref{fig:realizing+ringel}, the map can be visualized as contracting each shaded cell to a point, and each striped cell into a segment by contracting each stripe to a point on that segment.
\end{example}

\section{Elimination Systems}
\label{sec:elim_system}
In order to interpolate between fine mixed subdivisions and oriented matroid covectors, we consider a generalization of the set
of forests arising from a fine mixed subdivision which we call an \emph{elimination system}. The main result in this section is Theorem~\ref{thm:elim_system_poset_admits_factorization}, which states that a particular poset quotient associated to an elimination system admits a factorization into elementary quotients, as defined in Section~\ref{sec:poset_lattice_quotients}.

\subsection{Elimination Systems and their posets}

For a subgraph $F\subseteq\rows\times\ground$ of the complete bipartite
graph $K_{\rows,\ground}$ and $j\in\ground$, define the neighbourhood
$F_{j}:=\left\{ i:(i,j)\in F\right\} $. 

\begin{definition} \label{def:elimination+systems}

Let $\fmixed$ be a collection of subsets of $\rows\times\ground$.
Then $\fmixed$ is an \emph{elimination system} provided:
\begin{itemize}
\item[(E1)] For each $F\in\fmixed$ and for each $j\in\ground$, $F_{j}$ is
non-empty.
\item[(E2)] If $F\subseteq G\in\fmixed$ and $F_{j}$ is non-empty for all $j\in\ground$,
then $F\in\fmixed$.
\item[(E3)] If $F,G\in\fmixed$ and $j\in\ground$, then there exists $H\in\fmixed$
such that $H_{j}=F_{j}\cup G_{j}$ and $H_{k}\in\left\{ F_{k},G_{k},F_{k}\cup G_{k}\right\} $
for all $k\in\ground$ with $k\neq j$.
\end{itemize}
\end{definition}

Elimination systems are the same as generic tropical oriented matroids except without the comparability axiom; see \cite[Definition~3.5]{ArdilaDevelin:2009}.

Generalizing the face poset of the polyhedral complex of Proposition~\ref{prop:S_octagon} subdividing the boundary of $\octagon_{d}:=\cube_{d}+n\crosspolytope_{d}$, we introduce a poset associated with an elimination system. 

\begin{definition}\label{def:elim_poset}

Given an elimination system $\fmixed$, we define the following poset:
\[
\poset(\fmixed):=\left\{ (S,F):S\in\{-1,0,1\}^{\rows},\;F\in\fmixed,\;\supp(S)\supseteq\supp_{\rows}(F)\right\} .
\]
Recall from Proposition~\ref{prop:S_octagon} that $\supp_{\rows}(F)$
denotes those $i\in\rows$ such that $(i.j)\in F$ for at least one
$j\in\ground$. The ordering of the poset $\poset(\fmixed)$ is given as follows: $(S,F)\leq(T,G)$
if and only if $S\leq T$ and $F\subseteq G$. Recall that here $S\leq T$ means
that $S$ is obtained from $T$ by setting some entries to zero. For
example, $0-0+\leq+--+$; another way to see it is that the orthant
labeled by $S$ is contained in the orthant labeled by $T$. 

\end{definition}

\subsection{An Equivalence Relation of $\protect\poset(\protect\fmixed)$}

Let $\Pi$ be a partition of a finite
set $\mathsf{K}$. We say that two sign vectors $X,Y\in\{-1,0,1\}^{\mathsf{K}}$
are \emph{equivalent} (with respect to $\Pi$), and write $X\sim Y$,
if for all $s\in\{-,+\}$ and $\pi\in\Pi$, we have $X^{s}\cap\pi$
is nonempty iff $Y^{s}\cap\pi$ is nonempty. For example, the following
two sign vectors are equivalent with respect to the indicated partition
of the coordinates:

{\renewcommand{\arraystretch}{1.3}
\begin{table}[H]
\centering{}%
\begin{tabular}{c|c|c|ccc|ccc|ccccc|}
\cline{2-14} 
$\phantom{X\eqrel=Y}X:$ & $0$ & $+$ & $0$ & $-$ & $0$ & $0$ & $+$ & $-$ & $0$ & $0$ & $0$ & $+$ & $+$\tabularnewline
\cline{2-14} 
$\phantom{X\eqrel=Y}Y:$ & $0$ & $+$ & $-$ & $0$ & $0$ & $-$ & $+$ & $0$ & $+$ & $+$ & $0$ & $+$ & $0$\tabularnewline
\cline{2-14} 
\end{tabular}
\end{table}

}

This defines an equivalence relation on $\{-1,0,1\}^{\mathsf{K}}$.
We may think of each equivalence class $X\eqrel$ of this equivalence
relation as a sign vector in $\left\{ 0,+,-,\pm\right\} ^{\Pi}$.
For the above example, this would look like

{\renewcommand{\arraystretch}{1.3}
\begin{table}[H]
\centering{}%
\begin{tabular}{c|c|c|ccc|ccc|ccccc|}
\cline{2-14} 
$X\eqrel\:=Y\eqrel:$ & $0$ & $+$ & $\phantom{-}$ & $-$ & $\phantom{0}$ & $\phantom{-}$ & $\pm$ & $\phantom{-}$ & $\phantom{0}$ & $\phantom{+}$ & $+$ & $\phantom{+}$ & $\phantom{+}$\tabularnewline
\cline{2-14} 
\end{tabular}
\end{table}

}

Recall the construction of the sign matrix $SA_{F}$ associated with a sign vector $S$ and a graph $F$ on $\rows\sqcup\ground$ from Definition~\ref{def:signed+forest}. 
We introduce an equivalence relation $\sim_{A}$ based on the set of signs in each column of the sign matrix $SA_{F}$.

\begin{definition}\label{def:P(S)_homog_quotient}

Let $\Pi:=\left\{ \rows\times\{j\}:j\in\ground\right\} $ be a partition
of the edges of $K_{\rows,\ground}$. Define the following equivalence
relation $\sim_{A}$ on $\poset(\fmixed)$: Given $(S,F)$ and $(T,G)$
in $\poset(\fmixed)$, we say that $(S,F)\sim_{A}(T,G)$ if $S=T$
and $SA_{F}\sim SA_{G}$ with respect to the partition of $\rows\times\ground$
given by $\Pi$.

\end{definition}

\begin{example} \label{ex:quotients+signed+forests} Depicted below
are four elements from the poset $\poset(\fmixed)$ for Example~\ref{ex:small_patchworking}.
We show each element $(S,F)$ as $(S,SA_{F})$, noting that $F=\supp(SA_{F})$: 

  \begin{equation*}
    \begin{aligned}
      (S_1,S_1 A_{F_1}) = 
      \left(\begin{pmatrix} -\\+\\+ \end{pmatrix},
      \begin{pmatrix}
        + & - & + \\
        + & 0 & 0 \\
        - & 0 & 0
      \end{pmatrix}
      \right)\!,\, &
      (S_2,S_2 A_{F_2}) = 
      \left(\begin{pmatrix} -\\+\\+ \end{pmatrix},
      \begin{pmatrix}
        0 & - & + \\
        + & 0 & 0 \\
        - & - & 0
      \end{pmatrix}
      \right)\!,\, \\
      (S_3,S_3 A_{F_3}) = 
       \left(\begin{pmatrix} -\\+\\+ \end{pmatrix},
      \begin{pmatrix}
        0 & 0 & + \\
        + & 0 & - \\
        - & - & 0
      \end{pmatrix}
      \right)\!,\, &
      (S_4,S_4 A_{F_4}) = 
      \left(\begin{pmatrix} +\\+\\+ \end{pmatrix},
      \begin{pmatrix}
        0 & 0 & - \\
        + & 0 & - \\
        - & - & 0
      \end{pmatrix}
      \right)\!.
    \end{aligned}
  \end{equation*}
Observe that these four sign vectors correspond to four full-dimensional
cells in Figure~\ref{fig:mixed+subdivision+labeled}, of which three
are in the lower right orthant and the last is in the upper right
orthant. They correspond to cells following the red pseudoline in
Figure~\ref{fig:mixed+subdivision+PL+curves}, starting from the
triangle in the lower right orthant. We see right away that $(S_{4},F_{4})\not\sim_{A}(S_{\ell},F_{\ell})$
for $\ell=1,2,3$ as they differ in the first component. To check
for the equivalence of the other three pairs, we can consider the
image of the columns of $S_{1}A_{F_{1}},S_{2}A_{F_{2}},S_{3}A_{F_{3}}$
to $\{0,+,-,\pm\}^{3}$ as indicated before Definition~\ref{def:P(S)_homog_quotient}. 
This yields the three vectors $(\pm,-,+)$, $(\pm,-,+)$, $(\pm,-,\pm)$.
Hence, we get $(S_{1},F_{1})\sim_{A}(S_{2},F_{2})\not\sim_{A}(S_{3},F_{3})$.
\end{example}

\subsection{Properties of the Quotient $\protect\poset(\protect\fmixed)\eqrel_A$}

\label{subsec:Quotients_of_P(S)}
We assume we are given an elimination system $\fmixed$ on $\rows\times\ground$,
and a sign matrix $A\in\{-1,1\}^{\rows\times\ground}$. We denote
the poset $\poset(\fmixed)$ by~$\poset$. 

\begin{proposition}\label{prop:poset_covering}

Suppose $(S,F)$ is covered by $(T,G)$ in $\poset$. Then either
$F=G$ and $\left|S\right|=\left|T\right|-1$, or $S=T$ and $\left|F\right|=\left|G\right|+1$.

\end{proposition}

\begin{proof}The fact that $(S,F)\lneq(T,G)$ means that $S\leq T$
and $F\supseteq G$, and either $S\lneq T$ or $F\supsetneq G$. If
$S\lneq T$, then let $i\in\supp(T)\smallsetminus\supp(S)$. Then
$i\notin\supp(S)$, which means $i\notin\supp_{R}(F)$. Since $F\supseteq G$,
this means $i\notin\supp_{\rows}(G)$. Hence, $(T\smallsetminus i,G)$
is an element of $\poset$ such that
\[
(S,F)\leq(T\smallsetminus i,G)\leq(T,G).
\]
Since $(S,F)$ is covered by $(T,G)$, we conclude the first inequality
holds with equality, and hence $F=G$ and $\left|S\right|=\left|T\right|-1$.

Otherwise, $F\supsetneq G$. Let $(i,j)\in F\smallsetminus G$. Then
$(i,j)$ is not the only element of $F_{j}$, since otherwise we would
have $G_{j}=\emptyset$ which is forbidden by (E1). We therefore have
$(S,F\smallsetminus(i,j))\in\poset$ by (E2), and hence
\[
(S,F)\leq(S,F\smallsetminus(i,j))\leq(T,G).
\]
By covering, we conclude the second inequality holds with equality,
and hence $S=T$ and $\left|F\right|=\left|G\right|-1$. \end{proof}

\begin{corollary}\label{cor:poset_is_graded}

The poset $\poset$ is graded, with grading $\rank(S,F)=n+\left|S\right|-\left|F\right|$.\qed

\end{corollary}

Given two sign vectors $S,T\in\{-1,0,1\}^{\rows}$, define their \emph{intersection}
$S\cap T\in\{-1,0,1\}^{\rows}$ to be the sign vector such that{
}$(S\cap T)^{+}=S^{+}\cap T^{+}$ and $(S\cap T)^{-}=S^{-}\cap T^{-}$. 

\begin{proposition}\label{prop:augmented_poset_is_lattice}

The augmented poset $\lattice(\poset):=\poset\cup\{\hat{\zero},\hat{\oneone}\}$
is a lattice: if $(S,F),(T,G)\in\poset$ have a common lower bound,
then a greatest lower bound for both is given by $\left(S\cap T,F\cup G\right)$.

\end{proposition}

\begin{proof}Let $(S,F)$ and $(T,G)$ be elements of $\poset$ with
a common lower bound $(L,H)$. Then $H\supseteq F\cup G\supseteq F,G$
which implies by (E2) that $F\cup G\in\fmixed$. Similarly, we have
$L\leq S\cap T$ and so
\[
\supp_{\rows}(F\cup G)\subseteq\supp_{\rows}(H)\subseteq\supp(L)\subseteq\supp(S\cap T).
\]
We conclude $(S\cap T,F\cup G)\in\poset$ and is a lower bound of
$(S,F)$ and $(T,G)$. The fact that $H\supseteq F\cup G$ and $L\leq S\cap T$
shows that $(S\cap T,F\cup G)$ is in fact a greatest lower bound,
as $(L,H)$ was chosen arbitrarily.\end{proof}

Our next task is to generalize the equivalence relation $\sim_{A}$
on $\poset$ from Definition \ref{def:P(S)_homog_quotient}, by allowing
the partition $\Pi$ of $\rows\times\ground$ to vary. We assume fixed
a partition $\Pi$ of $\rows\times\ground$ which refines the partition
$\left\{ \rows\times\{j\}:j\in\ground\right\} $. In terms of this partition, we say $X\sim Y$
if $X^{s}\cap\pi$ is nonempty iff $Y^{s}\cap\pi$ is nonempty,
for all $s\in\{-,+\}$ and $\pi\in\Pi$. 

\begin{definition}\label{def:general_P(S)_homog_quotient}

For $(S,F),(T,G)\in\poset$, we say $(S,F)\sim_{A}(T,G)$ if and only
if $S=T$ and $SA_{F}\sim SA_{G}$. 

\end{definition}

\begin{proposition}\label{prop:eq_rel_is_P_homog}

The equivalence relation $\sim_{A}$ on $\poset$ is $\poset$-homogeneous.
In particular, $\poset\eqrel_{A}$ is a poset.

\end{proposition}

\begin{proof}

\global\long\def\restriction{\negmedspace\mid}
Let $(S,F)\leq(T,G)$ be two elements of $\poset$, and choose $(S,F')\sim_{A}(S,F)$.
Our goal is to find $G'\in\fmixed$ such that $(T,G')\in\poset$ and
$(S,F')\leq(T,G')\sim_{A}(T,G)$. Define
\begin{align*}
G' & :=\{(i,j)\in F':\text{if \ensuremath{\pi\in\Pi} contains \ensuremath{(i,j),} then there exists}\\
 & \phantom{{}:={}\{(i,j)\in F':}\text{\;\ensuremath{(\ell,j)\in\pi} such that \ensuremath{(TA_{G})_{\ell,j}=(SA_{F'})_{i,j}}}\}.
\end{align*}
Thus $(S,F')\leq(T,G')$. The definition of $G'$ ensures that every
sign appearing in the restricted sign vector $TA_{G'}\restriction_{\pi}$
also appears in $TA_{G}\restriction_{\pi}$, for all $\pi\in\Pi$.
Conversely, if $\pi\in\Pi$ and $(TA_{G})_{\ell,j}$ is nonzero for
some $(\ell,j)\in\pi$, then $SA_{F}\sim SA_{G}=TA_{G}$ implies there
exists $(i,j)\in\pi$ such that $(TA_{G})_{\ell,j}=(SA_{F})_{\ell,j}=(SA_{F'})_{i,j}$,
and therefore $TA_{G'}\restriction_{\pi}$ contains the sign $(TA_{G})_{\ell,j}$.
Note that we are using here the fact that $\Pi$ refines the partition
$\left\{ \rows\times\{j\}:j\in\ground\right\} $. We conclude $TA_{G}\sim TA_{G'}$.

Observe that $G_{j}$ is nonempty for every $j\in\ground$ by (E1),
and since $TA_{G}\sim TA_{G'}$ we also have $G'_{j}$ is nonempty
for every $j\in\ground$. Therefore, since $G'\subseteq F'$, we have
by (E2) that $G'\in\fmixed$. Moreover, $\supp_{\rows}(G')\subseteq\supp_{\rows}(F')\subseteq\supp(S)\subseteq\supp(T),$
so that $(T,G')\in\poset$. We conclude $(T,G')\sim_{A}(T,G)$. \end{proof}

For a generalized sign vector $X\eqrel\;\in\{0,+,-,\pm\}^{\Pi}$,
let $\left|X\eqrel\right|$ count the number of nonzero coordinates
in $X\eqrel$, with each $\pm$ counted twice. For example, if $X\eqrel\;=(0,\pm,-,+,-,\pm)$
then $\left|X\eqrel\right|=7$. Note that if $\Pi$ is the singleton
partition, then $X\eqrel$ is an ordinary sign vector and $\left|X\eqrel\right|=\left|X\right|$.

\begin{proposition}\label{prop:poset_quotient_is_graded}

The poset $\poset\eqrel_{A}$ is graded, with grading
\[
\rho((S,F)\eqrel_{A})=n+\left|S\right|-\left|SA_{F}\eqrel\right|.
\]

\end{proposition}

\begin{proof}

Fix $(S,F)\in\poset$. First note that $(S,F)$ is a maximal
element in the equivalence class $(S,F)\eqrel_{A}$ if and only if
$\left|(SA_{F})^{s}\cap\pi\right|\leq1$ for all $s\in\{-,+\}$ and
all $\pi\in\Pi$. Indeed, choose any $(S,G)\sim_{A}(S,F)$. Then (E2)
implies that we may find $(S,H)\geq(S,G)$ inside $(S,F)\eqrel_{A}$
such that $\left|(SA_{H})^{s}\cap\pi\right|\leq1$ for all $s\in\{-,+\}$
and all $\pi\in\Pi$. In particular, this statement holds for the
maximal elements of $(S,F)\eqrel_{A}$.

Now, for every maximal element $(S,G)\sim_{A}(S,F)$, we have
\begin{align*}
\rank((S,G)\eqrel_{A}) & =n+\left|S\right|-\left|SA_{G}\eqrel\right|\\
 & =n+\left|S\right|-\sum_{\pi\in\Pi}\left(\left|(SA_{G})^{+}\cap\pi\right|+\left|(SA_{G})^{-}\cap\pi\right|\right)\\
 & =n+\left|S\right|-\left|G\right|\\
 & =\rank(S,G).
\end{align*}

It remains to show that $\rho$ respects the covering relations. Suppose
that $(S,F)\eqrel_{A}$ is covered by $(T,G)\eqrel_{A}$ in $\poset\eqrel_{A}$.
By homogeneity, we may choose representatives $(S,F)$ and $(T,G)$
so that $(S,F)$ is covered by $(T,G)$ in $\poset$. Such an element
$(S,F)$ is necessarily a maximal element of the equivalence class
$(S,F)\eqrel_{A}$, which implies $\left|(SA_{F})^{+}\cap\pi\right|\leq1$
and $\left|(SA_{F})^{-}\cap\pi\right|\leq1$ for all $\pi\in\Pi$.
Since $(S,F)<(T,G)$, we have $TA_{G}=SA_{G}\leq SA_{F}$, and hence
$\left|(TA_{G})^{+}\cap\pi\right|\leq1$ and $\left|(TA_{G})^{-}\cap\pi\right|\leq1$
for all $\pi\in\Pi$. It follows $(T,G)$ is maximal in $(T,G)\eqrel_{A}$.
We conclude
\[
\rank((S,F)\eqrel_{A})=\rank(S,F)=\rank(T,G)-1=\rank((T,G)\eqrel_{A})-1.\qedhere
\]
\end{proof}

\begin{proposition}\label{prop:augmented_poset_quotient_is_a_lattice}

The augmented poset $\lattice(\poset\eqrel_{A})$ is a lattice. 

\end{proposition}\global\long\def\signs{\mathrm{signs}}

\begin{proof}Choose $(S,F)\eqrel_{A}$ and $(T,G)\eqrel_{A}$ with
a common lower bound in $\poset\eqrel_{A}$. By homogeneity and Proposition~\ref{prop:augmented_poset_is_lattice},
we may choose the representatives $(S,F)$ and $(T,G)$ so that $(S\cap T,F\cup G)\in\poset$.
By homogeneity, then, $(S\cap T,F\cup G)\eqrel_{A}$ is a lower bound
for both $(S,F)\eqrel_{A}$ and $(T,G)\eqrel_{A}$.

We show this is a greatest lower bound. Given a lower bound $(L,H)\eqrel_{A}$,
we may find $(S,F')\sim_{A}(S,F)$ and $(T,G')\sim_{A}(T,G)$ such
that $(L,H)\leq(S,F')$ and $(L,H)\leq(T,G')$ in $\poset$. Hence,
by Proposition~\ref{prop:augmented_poset_is_lattice}, $(L,H)\leq(S\cap T,F'\cup G')\in\poset$.
Therefore, it suffices to show
\[
(S\cap T,F'\cup G')\sim_{A}(S\cap T,F\cup G).
\]
For all $\pi\in\Pi$ and $s\in\{-,+\}$, we have
\begin{align*}
((S\cap T)A_{F'\cup G'})^{s}\cap\pi\text{ nonempty} & \iff((SA_{F'})^{s}\cup(TA_{G'})^{s})\cap\pi\text{ nonempty}\\
 & \iff((SA_{F})^{s}\cup(TA_{G})^{s})\cap\pi\text{ nonempty}\\
 & \iff((S\cap T)A_{F\cup G})^{s}\cap\pi\text{ nonempty.}
\end{align*}
In particular, this shows $(S\cap T,F'\cup G')\sim_{A}(S\cap T,F\cup G)$.\end{proof}

We now come to the main theorem of this section:

\begin{theorem}\label{thm:elim_system_poset_admits_factorization}

The poset $\poset(\fmixed)\eqrel_{A}$ admits a factorization $\poset(\fmixed)=\poset_{0},\poset_{1},\ldots,\poset_{k}=\poset(\fmixed)\eqrel_{A}$
into elementary quotients, such that the augmented poset $\lattice(\poset_{i})$
is a lattice for each $i=0,1,\ldots,k-1$.

\end{theorem}

\begin{proof}

By Proposition \ref{prop:augmented_poset_is_lattice}, $\lattice(\poset)$
is a lattice. Thus, let $\bar{\Pi}$ be a partition of $\rows\times\ground$
which refines the partition $\Pi$ and has at least one part $\pi\in\bar{\Pi}$
such that $\left|\pi\right|\geq2$. Let $e:=(i,j)\in\pi$, and let
$\ddot{\Pi}$ be the refinement of $\bar{\Pi}$ obtained by splitting
the part $\pi$ into two parts: $\{e\}$ and $\pi\smallsetminus\{e\}$.
That is,
\[
\ddot{\Pi}=(\bar{\Pi}\smallsetminus\{\pi\})\cup\{\{e\},\pi\smallsetminus\{e\}\}.
\]
\global\long\def\simddot{\ddot{\sim}}
\global\long\def\simbar{\bar{\sim}}
\global\long\def\simddot{\raisebox{-1pt}{$\ddot{\raisebox{1pt}{$\sim$}}$}} \global\long\def\simbar{\raisebox{-1pt}{$\bar{\raisebox{1pt}{$\sim$}}$}}\global\long\def\eqrelddot{/\simddot}
\global\long\def\eqrelbar{/\simbar}
Let $\simddot$ and $\simbar$ denote the equivalence relations on
sign vectors on $\rows\times\ground$ induced by $\ddot{\Pi}$ and
$\bar{\Pi}$, respectively. These determine $\poset$-homogeneous
equivalence relations $\simddot_{A}$ and $\simbar_{A}$ by Proposition
\ref{prop:eq_rel_is_P_homog}. Let $\ddot{\poset}=\poset\eqrelddot_{A}$.
Since $\simbar_{A}$ is $\poset$-homogeneous, and since $\simddot_{A}$
refines $\simbar_{A}$, we have that $\simbar_{A}$ is $\ddot{\poset}$-homogeneous.
Moreover, there is a natural identification $\ddot{\poset}\eqrelbar_{A}=\poset\eqrelbar_{A}$.
Therefore, by induction, the theorem is proved if we can show that
$\ddot{\poset}\eqrelbar_{A}$ is an elementary quotient whose augmented
poset is a lattice. In fact the lattice assertion follows from Proposition
\ref{prop:augmented_poset_quotient_is_a_lattice}.

Fix $(S,F)\in\poset$. We would like to show that the equivalence
class containing $(S,F)\eqrelddot_{A}$ in $\ddot{\poset}\eqrelbar_{A}$
is either a singleton, or consists of exactly three elements two of
which cover a third. Note that if $SA_{F}\restriction_{\pi}$ is the
zero vector, then this equivalence class is indeed a singleton. This
is because we would immediately know that $(SA_{F})_{e}=0$ and $SA_{F}\restriction_{\pi\smallsetminus e}=\zero$,
hence in this case $(S,F)\eqrelddot_{A}$ is completely determined
by $(S,F)\eqrelbar_{A}$.

Otherwise, the sign vector $SA_{F}\restriction_{\pi}$ is non-zero,
and in this case there are exactly three generalized sign vectors
$X_{1},X_{2},X_{3}\eqrelddot\,\in\{0,-,+,\pm\}^{\ddot{\Pi}}$, depending
on $SA_{F}\restriction_{\pi}$ and $(SA)_{e}$, such that
$X_{1}\ \simbar\ X_{2}\ \simbar\ X_{3}\ \simbar\ SA_{F}.$ The restrictions
of these to $\pi$ are depicted below, in all of four possible cases:

{\renewcommand{\arraystretch}{1.3}

\begin{table}[H]
\centering{}%
\begin{tabular}{r|c|c|c|c|c|c}
\cline{2-4} \cline{6-6} 
$e$ & $0$ & $+$ & $+$ & \multirow{2}{*}{$\rightarrow$} & \multirow{2}{*}{$\pm$} & \multirow{2}{*}{$\pi$}\tabularnewline
\cline{2-4} 
$\stackrel[\phantom{|}]{\phantom{|}}{\pi\smallsetminus e}$ & $\pm$ & $-$ & $\pm$ &  &  & \tabularnewline
\cline{2-4} \cline{6-6} 
\multicolumn{1}{r}{} & \multicolumn{1}{c}{$X_{1}$} & \multicolumn{1}{c}{$X_{2}$} & \multicolumn{1}{c}{$X_{3}$} & \multicolumn{3}{c}{$\ SA_{F}$}\tabularnewline
\end{tabular}$\quad$%
\begin{tabular}{r|c|c|c|c|c|c}
\cline{2-4} \cline{6-6} 
$e$ & $0$ & $-$ & $-$ & \multirow{2}{*}{$\rightarrow$} & \multirow{2}{*}{$\pm$} & \multirow{2}{*}{$\pi$}\tabularnewline
\cline{2-4} 
$\stackrel[\phantom{|}]{\phantom{|}}{\pi\smallsetminus e}$ & $\pm$ & $+$ & $\pm$ &  &  & \tabularnewline
\cline{2-4} \cline{6-6} 
\multicolumn{1}{r}{} & \multicolumn{1}{c}{$X_{1}$} & \multicolumn{1}{c}{$X_{2}$} & \multicolumn{1}{c}{$X_{3}$} & \multicolumn{3}{c}{$\ SA_{F}$}\tabularnewline
\end{tabular}\bigskip{}
\begin{tabular}{r|c|c|c|c|c|c}
\cline{2-4} \cline{6-6} 
$e$ & $0$ & $+$ & $+$ & \multirow{2}{*}{$\rightarrow$} & \multirow{2}{*}{$+$} & \multirow{2}{*}{$\pi$}\tabularnewline
\cline{2-4} 
$\stackrel[\phantom{|}]{\phantom{|}}{\pi\smallsetminus e}$ & $+$ & $0$ & $+$ &  &  & \tabularnewline
\cline{2-4} \cline{6-6} 
\multicolumn{1}{r}{} & \multicolumn{1}{c}{$X_{1}$} & \multicolumn{1}{c}{$X_{2}$} & \multicolumn{1}{c}{$X_{3}$} & \multicolumn{3}{c}{$\ SA_{F}$}\tabularnewline
\end{tabular}$\quad$%
\begin{tabular}{r|c|c|c|c|c|c}
\cline{2-4} \cline{6-6} 
$e$ & $0$ & $-$ & $-$ & \multirow{2}{*}{$\rightarrow$} & \multirow{2}{*}{$-$} & \multirow{2}{*}{$\pi$}\tabularnewline
\cline{2-4} 
$\stackrel[\phantom{|}]{\phantom{|}}{\pi\smallsetminus e}$ & $-$ & $0$ & $-$ &  &  & \tabularnewline
\cline{2-4} \cline{6-6} 
\multicolumn{1}{r}{} & \multicolumn{1}{c}{$X_{1}$} & \multicolumn{1}{c}{$X_{2}$} & \multicolumn{1}{c}{$X_{3}$} & \multicolumn{3}{c}{$\ SA_{F}$}\tabularnewline
\end{tabular}
\end{table}

}

The following argument applies simultaneously to all four cases shown
above. Suppose there are at least two distinct elements $(S,F)\eqrelddot_{A}$
and $(S,F')\eqrelddot_{A}$ in the same equivalence class of $\ddot{\poset}\eqrelbar_{A}$.
Then there exists a unique $i\in\{1,2,3\}$ such that
\[
\{SA_{F}\eqrelddot,\,SA_{F'}\eqrelddot,\,X_{i}\eqrelddot\}=\{X_{1}\eqrelddot,\,X_{2}\eqrelddot,\,X_{3}\eqrelddot\}.
\]
for some $i\in\{1,2,3\}$. We consider the three cases separately. 
\begin{itemize}
\item If $i=1$ or $2$, then without loss of generality assume $SA_{F}\ \simddot\ X_{3}$. 
\begin{itemize}
\item If $i=1$, then by (E2) the set $F''=F\smallsetminus e$ is in $\fmixed$,
and $SA_{F''}\ \simddot\ X_{1}$. 
\item If $i=2$, then by (E2) the set $F''=F\smallsetminus((SA_{F})^{s}\cap\pi)$
is in $\fmixed$, where $s$ is the unique sign appearing in $X_{3}\restriction_{\pi\smallsetminus e}$
but not $X_{2}\restriction_{\pi\smallsetminus e}$, and $SA_{F''}\ \simddot\ X_{2}$.
\end{itemize}
\item If $i=3$, then by (E3), we can find $F''\in\fmixed$ such that
\begin{align*}
SA_{F''}\restriction_{\pi} & =SA_{F\cup F'}\restriction_{\pi}\\
SA_{F''}\restriction_{\tau} & \in\left\{ SA_{F}\restriction_{\tau},\;SA_{F'}\restriction_{\tau},\;SA_{F\cup F'}\restriction_{\tau}\right\} \text{ for all \ensuremath{\tau\in\bar{\Pi}\smallsetminus\pi}.}
\end{align*}
 This shows $SA_{F''}\ \simddot\ X_{3}$. We remark that this is the only time (E3) is used.
\end{itemize}
In all three cases, we therefore have found $(S,F'')\ \simbar_{A}\ (S,F)$
such that $SA_{F''}\ \simddot\ X_{i}$. Therefore the equivalence
class of $(S,F)\eqrelddot_{A}$ in $\ddot{\poset}\eqrelbar_{A}$ consists
of the three distinct elements $(S,F)\eqrelddot_{A}$ , $(S,F')\eqrelddot_{A}$,
and $(S,F'')\eqrelddot_{A}$. Their gradings in $\ddot{\poset}=\poset\eqrelddot$
are given by, by Proposition \ref{prop:poset_quotient_is_graded},
$n-\left|S\right|-\left|X_{i}\eqrelddot\right|$ for $i=1,2,3$. Inspecting
the above four tables, we conclude that two of these elements cover
the third in $\ddot{\poset}$.\end{proof}

%\todo[inline]{YCH: I suggest we elaborate with more non-realizable examples, e.g. Jesus's triangulation}

%\section{Conclusion} \label{sec:conclusion}

%With a new point of view and machinery, we extend the connections between several prominent objects in matroid theory, discrete geometry, and tropical geometry beyond the ``realizable'' territory.
%We also initiate the study of the interaction between matroid subdivisions and signs, or more generally, hyperfields.
%As mentioned throughout in our paper, many of these ideas are interesting in its own right and could potentially be applied to other settings.
%Recall that the original motivation of matching fields comes from combinatorial commutative algebra and the corresponding algebraic geometry, it is also interesting to see if something can be said in these areas.
%We end with a few more open problems, focusing on aspects that were not fully discussed in the main content.

\section{Quotients of Regular Cell Complexes}
\label{sec:quotients_reg_cell_cpxs}

\subsection{Background: Regular Cell Complexes and PL Topology}
\label{sec:basics+cell+complexes}

We quickly review the key aspects of combinatorial
topology we wish to use. The main reference
here is \cite[Section 4.7]{BLSWZ:1993}.

\subsubsection{Regular Cell Complexes}
\begin{definition}

A \emph{regular cell complex} $\Delta$ is a Hausdorff space $\left\Vert \Delta\right\Vert $
together with a finite collection of balls $\Delta$ such that:
\begin{enumerate}
\item The interiors of the balls in $\Delta$ partition the space: $\left\Vert \Delta\right\Vert =\uni_{\sigma\in\Delta}\sigma^{\circ}$.
\item The boundary of any $\sigma\in\Delta$ is a union of members of $\Delta$:
$\bd(\sigma)=\uni_{\tau\subset\sigma}\tau$.
\end{enumerate}
\end{definition}

\begin{definition}

An important special case of the above definition is a \emph{polyhedral
cell complex}. This is a regular cell complex $\Delta$ such that
each $\sigma\in\Delta$ is a polytope in $\RR^{d}$, and for each
$\sigma,\tau\in\Delta$ we have $\sigma\cap\tau$ is a face of both
$\sigma$ and $\tau$.
If every polytope in $\Delta$ is a simplex,
we call $\Delta$ a \emph{geometric simplicial complex}.
A \emph{triangulation} of a set $Q \subset \RR^d$ is a geometric simplicial complex with underlying space $Q$. 

\end{definition}

\begin{definition}

The \emph{face poset} $\poset(\Delta)$ of a regular cell complex
$\Delta$ is the poset whose underlying set is the set of balls $\Delta$,
and whose ordering is given by inclusion.

\end{definition}

\begin{definition}\label{order_complex}

The\emph{ order complex} $\Delta(\poset)$ of a poset $\poset$ is
the simplicial complex whose vertices are the elements of $\poset$
and whose simplices are the chains of $\poset$. We denote by $\norm{\poset}$
the topological space $\norm{\Delta(\poset)}$.

\end{definition}

\begin{proposition}

Every abstract simplicial complex (i.e. set system closed under taking
subsets) can be realized as a geometric simplicial complex in some
Euclidean space.

\end{proposition}

\subsubsection{PL Balls and Spheres}

\begin{definition}

Given $P\subset\RR^{k}$, $Q\subset\RR^{\ell}$, a map $f:P\rightarrow Q$
is \emph{piecewise linear} (PL) if there is a triangulation $\Delta$
of $P$ into simplices such that $f$ restricted to each simplex of
$\Delta$ is an affine function. That is, if $\sigma=\conv(v_{0},\ldots,v_{k})\in\Delta$
then $f|_{\sigma}$ satisfies 
\[
f(\lambda_{0}v_{0}+\lambda_{1}v_{1}+\cdots+\lambda_{k}v_{k})=\lambda_{0}f(v_{0})+\lambda_{1}f(v_{1})+\cdots+\lambda_{k}f(v_{k})
\]
for all convex combinations $\sum_{i}\lambda_{i}v_{i}$ of the vertices
$v_{0},\ldots,v_{k}$ of $\sigma$. We call a PL map that is also
a homeomorphism a \emph{PL homeomorphism}.

\end{definition}

\begin{definition}

Let $P\subset\RR^{k}$ be the underlying space of a polyhedral cell complex. Then $P$ is a \emph{PL $d$-sphere} (resp.
\emph{PL $d$-ball}) if there is a PL homeomorphism from $P$ to the
boundary of the standard $d$-simplex (resp. to the standard $d$-simplex).

\end{definition}

\begin{proposition}\label{prop:PL_ball_facts}$\;$
\begin{enumerate}
\item \cite[Theorem~4.7.21(i)]{BLSWZ:1993} The union of two PL d-balls,
whose intersection is a PL $(d-1)$-ball lying in the boundary of
each, is a PL $d$-ball.
\item \cite[Theorem~4.7.21(ii)]{BLSWZ:1993} The union of two PL d-balls,
which intersect along their entire boundaries, is a PL $d$-sphere.
\item \cite[Theorem~4.7.21(iii)]{BLSWZ:1993} (Newman's Theorem) The closure
of the complement of a PL $d$-ball embedded in a PL $d$-sphere is
itself a PL $d$-ball.
\end{enumerate}
\end{proposition}

\begin{lemma}\label{prop:union of PL balls is PL ball} Let $\sigma,\tau$
be two PL $d$-balls, such that $\sigma\cap\tau$ is a PL $(d-1)$-ball
contained in the boundaries of both $\sigma$ and $\tau$.
Then the interior of $\sigma\cup\tau$ is equal to $\sigma^{\circ}\cup\tau^{\circ}\cup\left(\sigma\cap\tau\right)^{\circ}$.

\end{lemma}

\begin{proof} By Proposition~\ref{prop:PL_ball_facts} (1), $\sigma\cup\tau$
is a PL $d$-ball. We start by showing that $\sigma^{\circ}$ contains
$(\sigma\cup\tau)^{\circ}\backslash\tau$. Let $x\in(\sigma\cup\tau)^{\circ}\backslash\tau$.
Then there is an open set $\mathcal{U}\subset(\sigma\cup\tau)^{\circ}$
containing $x$ and a homeomorphism $\varphi:\mathcal{U}\rightarrow B_{d}^{\circ}\subset\RR^{d}$
sending $x$ to $\zero$. Here $B_{d}$ denotes the ball of radius
1 in $\RR^{d}$ centred at the origin. Since $\tau$ is closed in
$\sigma\cup\tau$, and since $x\notin\tau$, we further have that
$\varphi(\mathcal{U}\backslash\tau)$ is an open set which contains
the origin; hence there exists $\delta>0$ such that the scaled open
ball $\delta\cdot B_{d}^{\circ}$ is contained in $\varphi(\mathcal{U}\backslash\tau)=\varphi(\mathcal{U})\backslash\varphi(\tau)$.
It follows that $\varphi^{-1}(\delta\cdot B_{d}^{\circ})$ is an open
neighbourhood of $x$, homeomorphic to $B_{d}^{\circ}$, and entirely
contained in $\sigma$. In particular, this means that $x\in\sigma^{\circ}$.
We conclude $\sigma^{\circ}\supseteq(\sigma\cup\tau)^{\circ}\backslash\tau$.

From this containment we immediately get
\[
\partial\sigma\subseteq\sigma\backslash((\sigma\cup\tau)^{\circ}\backslash\tau)=(\sigma\cap\tau)\cup(\sigma\cap\partial(\sigma\cup\tau)),
\]
and in particular
\[
U:=\partial\sigma\backslash(\sigma\cap\tau)\subseteq\partial(\sigma\cup\tau).
\]
Since boundaries are closed, $V:=\partial\left(\sigma\cup\tau\right)\cap\partial\sigma$
is closed inside $\partial\sigma$. Now $W:=\partial\left(\sigma\cap\tau\right)$
is the boundary of $U$ in $\partial\sigma$, thus it is contained
in $\overline{U}\subset V\subset\partial\left(\sigma\cup\tau\right)$.
Similarly, $U':=\partial\tau\setminus\left(\sigma\cap\tau\right)\subset\partial\left(\sigma\cup\tau\right)$.
By Proposition~\ref{prop:PL_ball_facts} (3), both $U\cup W,U'\cup W$
are PL $(d-1)$-balls with common boundary $W$, so by Proposition~\ref{prop:PL_ball_facts}
(2), $U\cup W\cup U'$ is a PL $(d-1)$-sphere contained in $\partial\left(\sigma\cup\tau\right)$.
Invariance of Domain implies the containment is an equality, see for
example {\cite[Corollary 2B.4]{Hatcher:2002}}. After taking the complement with respect to $\sigma\cup\tau$, this equality yields an expression for $(\sigma\cup\tau)^{\circ}$ which simplifies to $\sigma^{\circ}\cup\tau^{\circ}\cup\left(\sigma\cap\tau\right)^{\circ}$.\end{proof} 

\subsubsection{Regular Cell Complexes that are PL Spheres}

\begin{definition}We say that a regular cell complex $\Delta$ with
face poset $\poset$ is a \emph{PL sphere} if some realization of
the order complex $\Delta(\poset)$ in some Euclidean space is a PL
sphere.

\end{definition}

\begin{proposition}[{\cite[Proposition 4.7.26(iii)]{BLSWZ:1993}}]

\label{prop:pl_sphere_cells_are_pl_balls}Let $\Delta$ be a regular
cell complex that is a PL sphere. Then every $\sigma\in\Delta$ is
a PL ball.

\end{proposition}

An important fact about PL spheres is that they admit a dual cell
structure:

\begin{proposition}[{\cite[Proposition 4.7.26(iv)]{BLSWZ:1993}}]

\label{prop:dual cell cpx}Let $\Delta$ be a regular cell complex
that is a PL sphere. Then there exists a regular cell complex $\Delta^{\dual}$,
also a PL sphere, such that $\left\Vert \Delta\right\Vert =\left\Vert \Delta^{\dual}\right\Vert $
and $\poset(\Delta^{\dual})\simeq\poset(\Delta)^{\dual}$.

\end{proposition}

Here $\poset^{\dual}$ denotes the dual poset of $\poset$. In the
special case when $\Delta$ is a polyhedral cell complex, there is
a non-canonical way to construct this $\Delta^{\dual}$:

\begin{definition}

Let $\Delta$ be a polyhedral cell complex. A \emph{first derived
subdivision} $\Delta^{1}$ is a subdivision of $\Delta$ obtained
as follows: choose a point $x_{\sigma}$ in the relative interior
of each $\sigma\in\Delta$. Then, $\Delta^{1}$ is given by
\[
\Delta^{1}:=\left\{ \conv(x_{\sigma_{1}},\ldots,x_{\sigma_{k}}):\sigma_{1}\subsetneq\sigma_{2}\subsetneq\cdots\subsetneq\sigma_{k},\;\text{each \ensuremath{\sigma_{i}\in\Delta}}\right\} .
\]

\end{definition}

\begin{theorem}[{\cite[\S~1.6]{Hudson:1967}}]

If $\Delta$ is a polyhedral cell complex then $\Delta^{\dual}$ may
be constructed as follows: Choose a first derived subdivision $\Delta^{1}$
of $\Delta$. For each cell $\sigma\in\Delta$, define
\[
\sigma^{\dual}:=\bigcap_{\text{\ensuremath{v} vertex of \ensuremath{\sigma}}}\left\Vert \overline{\str}(v;\Delta^{1})\right\Vert 
\]
where $\overline{\str}(\sigma;\Delta):=\left\{ \tau\in\Delta:\text{\ensuremath{\tau} is contained in a cell containing \ensuremath{\sigma}}\right\} $.
Then let 
\[
\Delta^{\dual}:=\left\{ \sigma^{\dual}:\sigma\in\Delta\right\} .
\]

\end{theorem}

\subsection{Quotients of Regular Cell Complexes}

\label{subsec:quotients_of_cell_complexes}
Our next goal is to develop
a notion of a quotient of a regular cell complex $\Delta$, in which
cells are merged together according to a given equivalence relation
on the cells of $\Delta$.

Let $\Delta$ be a regular cell complex with face poset $\poset$,
so that $\norm{\Delta}\subseteq\RR^{d}$. Given a homogeneous quotient
$\poset\eqrel$ of $\poset$, define the set 
\[
\Delta\eqrel\;:=\left\{ \uni\eqclass{\sigma}:\eqclass{\sigma}\in\poset\eqrel\right\} 
\]
where $\uni\eqclass{\sigma}$ denotes the union $\uni_{\tau\in\eqclass{\sigma}}\tau$.
Note that homogeneity of $\sim$ implies that $\uni\tilde{\sigma}\subseteq\uni\tilde{\tau}$
as sets if and only if $\uni\tilde{\sigma}\leq\uni\tilde{\tau}$ in
$\poset\eqrel$.

Under certain conditions, $\Delta\eqrel$ is again a regular cell
complex:

\begin{theorem}\label{thm:elem_quotient_implies_quotient_complex_is_regular}

Suppose:
\begin{enumerate}
\item The poset $\poset\eqrel$ is an elementary quotient,
\item The augmented poset $\lattice(\poset)$ is a lattice, and 
\item Each $\sigma\in\Delta$ is a PL ball. 
\end{enumerate}
Then $\Delta\eqrel$ is a regular cell complex with face poset $\poset\eqrel$,
such that each $\uni\tilde{\sigma}\in\Delta\eqrel$ is a PL ball.

\end{theorem}

\begin{corollary}\label{cor:factorization_implies_quotient_complex_is_regular}

Suppose $\poset$ admits a factorization $\poset=\poset_{0},\poset_{1},\ldots,\poset_{k}=\poset\eqrel$
into elementary quotients, such that $\lattice(\poset_{i})$ is a
lattice for each $i=0,1,2,\ldots,k-1$. Suppose further that each
$\sigma\in\Delta$ is a PL ball. Then $\Delta\eqrel$ is a regular
cell complex with face poset $\poset\eqrel$.

\end{corollary} In the remainder of this section we prove
Theorem~\ref{thm:elem_quotient_implies_quotient_complex_is_regular}. The main ingredient is a topological criterion for $\Delta\eqrel$ to be a regular cell complex:

\begin{lemma}

\label{lem: homog quotient implies cell cpx}

Suppose that each $\uni\eqclass{\sigma}$ in $\Delta\eqrel$ is a
ball whose interior equals the union of the interiors of the cells
of $\eqclass{\sigma}$. Then $\Delta\eqrel$ is a regular cell complex
with face poset $\poset\eqrel$.

\end{lemma}

\begin{proof}

We first show that $\Delta\eqrel$ is a regular cell complex. It is
clear that the underlying topological spaces of $\Delta$ and $\Delta\eqrel$
are the same. To see that the interiors of the balls in $\Delta\eqrel$
are disjoint, let $\uni\eqclass{\sigma}_{1}$ and $\uni\eqclass{\sigma}_{2}$
be two balls in $\Delta\eqrel$ such that 
\[
\left(\uni\eqclass{\sigma}_{1}\right)^{\circ}\cap\left(\uni\eqclass{\sigma}_{2}\right)^{\circ}=\left(\bigcup_{\tau_{1}\in\eqclass{\sigma}_{1}}\tau_{1}^{\circ}\right)\cap\left(\bigcup_{\tau_{2}\in\eqclass{\sigma}_{2}}\tau_{2}^{\circ}\right)=\bigcup_{\substack{\tau_{1}\in\eqclass{\sigma}_{1}\\
\tau_{2}\in\eqclass{\sigma}_{2}
}
}\tau_{1}^{\circ}\cap\tau_{2}^{\circ}
\]
is non-empty. In particular, there must exist $\tau_{1}\in\eqclass{\sigma_{1}}$
and $\tau_{2}\in\eqclass{\sigma_{2}}$ such that $\tau_{1}^{\circ}$
and $\tau_{2}^{\circ}$ intersect. This can only happen if $\tau_{1}=\tau_{2}$,
and hence $\eqclass{\sigma}_{1}=\eqclass{\sigma}_{2}$. To see that
the boundary of each $\uni\eqclass{\sigma}$ in $\Delta\eqrel$ is
a union of members of $\Delta\eqrel$, let $\uni\eqclass{\sigma}$
be an element of $\Delta\eqrel$. Then
\begin{equation}
\bigcup_{\eqclass{\tau}<\eqclass{\sigma}}\left(\uni\eqclass{\tau}\right)=\bigcup_{\delta\in\eqclass{\sigma}}\bigcup_{\substack{\tau<\delta\\
\tau\notin\eqclass{\sigma}
}
}\tau=\bigcup_{\delta\in\eqclass{\sigma}}\bigcup_{\substack{\tau<\delta\\
\tau\notin\eqclass{\sigma}
}
}\tau^{\circ}.\label{eq:boundary of cell-1}
\end{equation}
We justify the last equality. We may write $\tau=\uni_{\gamma\leq\tau}\gamma^{\circ}$
for every $\tau\in\Delta$. Hence, the last equality holds provided
we can show the following statement: whenever we have $\gamma\leq\tau<\delta\in\eqclass{\sigma}$
where $\tau\notin\eqclass{\sigma}$, we must also have $\gamma\notin\eqclass{\sigma}$.
The condition $\gamma\leq\tau$ implies $\eqclass{\gamma}\leq\eqclass{\tau}$.
The condition $\tau<\delta$ implies $\eqclass{\tau}\leq\eqclass{\delta}=\eqclass{\sigma}$.
On the other hand, the condition $\tau\notin\eqclass{\sigma}$ implies
$\eqclass{\tau}\neq\eqclass{\sigma}$, and therefore $\eqclass{\tau}<\eqclass{\sigma}$.
We conclude $\eqclass{\gamma}\leq\eqclass{\tau}<\eqclass{\sigma}$,
and in particular $\gamma\notin\eqclass{\sigma}$. Note that this
argument uses the fact that $\poset\eqrel$ is a poset, which follows
from homogeneity of $\sim$. Now, since the interiors of cells of
$\Delta$ partition $\norm{\Delta}$, we have by (\ref{eq:boundary of cell-1})
that
\begin{align*}
\bigcup_{\eqclass{\tau}<\eqclass{\sigma}}\left(\uni\eqclass{\tau}\right)=\bigcup_{\delta\in\eqclass{\sigma}}\left(\left(\bigcup_{\tau\leq\delta}\tau^{\circ}\right)\smallsetminus\bigcup_{\gamma\in\eqclass{\sigma}}\gamma^{\circ}\right) & =\bigcup_{\delta\in\eqclass{\sigma}}\left(\delta\smallsetminus\bigcup_{\gamma\in\eqclass{\sigma}}\gamma^{\circ}\right)=\left(\uni\eqclass{\sigma}\right)\smallsetminus\bigcup_{\gamma\in\eqclass{\sigma}}\gamma^{\circ}.
\end{align*}
We therefore conclude
\[
\bd\left(\uni\eqclass{\sigma}\right)=\left(\uni\eqclass{\sigma}\right)\smallsetminus\left(\uni\eqclass{\sigma}\right)^{\circ}=\left(\uni\eqclass{\sigma}\right)\smallsetminus\bigcup_{\gamma\in\eqclass{\sigma}}\gamma^{\circ}=\bigcup_{\eqclass{\tau}<\eqclass{\sigma}}\left(\uni\eqclass{\tau}\right).
\]

The proof that the face poset of $\Delta\eqrel$ is $\poset\eqrel$
is straightforward. If $\uni\eqclass{\tau}\subseteq\uni\eqclass{\sigma}$,
then this means in particular that $\tau\subseteq\sigma$, hence $\tau\leq\sigma$
in $\poset$, hence $\eqclass{\tau}\leq\eqclass{\sigma}$ in $\poset\eqrel$.
Conversely, if $\eqclass{\tau}\leq\eqclass{\sigma}$ in $\poset\eqrel$,
then there exists a cell of $\eqclass{\tau}$ contained in some cell
of $\eqclass{\sigma}$. By homogeneity, then, every cell of $\eqclass{\tau}$
in contained in some cell of $\eqclass{\sigma}$. Hence $\uni\eqclass{\tau}\subseteq\uni\eqclass{\sigma}$.\end{proof}

\begin{proposition}[{\cite[Section 4.7, pp.204]{BLSWZ:1993}}]

\label{prop:Lattice iff closed under non-empty intersections}

Let $\Delta$ be a regular cell complex with face poset $\poset$.
Then the augmented poset $\lattice(\poset)=\poset\cup\{\hat{\zero},\hat{\oneone}\}$
is a lattice if and only if $\Delta$ is closed under non-empty intersections:
for all $\sigma,\tau\in\Delta$ such that $\sigma\cap\tau$ is non-empty,
we have $\sigma\cap\tau\in\Delta$.\qed

\end{proposition}

\begin{proof}[Proof of Theorem~\ref{thm:elem_quotient_implies_quotient_complex_is_regular}]

It is clear that for any singleton class $\eqclass{\sigma}=\left\{ \sigma\right\} $,
$\uni\eqclass{\sigma}$ satisfies the hypothesis of Lemma \ref{lem: homog quotient implies cell cpx}.
Now suppose $\eqclass{\sigma}=\left\{ \sigma,\tau,\gamma\right\} $
is a class in $\sim$. It is known that the function $\sigma\mapsto\dim(\sigma)$
is a rank function on $\poset$. In particular, since $\sigma$ and
$\tau$ cover $\gamma$, then we must have $\dim(\sigma)=\dim(\tau)=\dim(\gamma)+1$.
Moreover, since $\lattice(\poset)$ is a lattice, we must have $\gamma=\sigma\cap\tau$
by Proposition \ref{prop:Lattice iff closed under non-empty intersections}.
Proposition \ref{prop:PL_ball_facts} (1) and Lemma \ref{prop:union of PL balls is PL ball}
then show that $\uni\eqclass{\sigma}$ is a PL ball which satisfies
the hypothesis of Lemma \ref{lem: homog quotient implies cell cpx}.
\end{proof}
\section{Proof of the Main Theorem} \label{sec:pf_main_theorem}
In this section we prove Theorem~\ref{thm:arrangement_of_pseudospheres_extended}. We assume the fine mixed subdivision $\fmixed$ of $n\ssimplex_{d-1}$, the sign matrix $A$, and the oriented matroid $\widetilde{\orientedmatroid}$ are
as defined in Section~\ref{sec:mixed_subdiv_to_PSA}.
\subsection{Properties of the Extended Patchworking Complex} \label{sec:properties+extended+patchworking+complex}
We begin this section by establishing some technical details of the extended patchworking complex defined in Section~\ref{sec:mixed_subdiv_to_PSA}. 
Note that each face of the polytope $n\ssimplex_{d-1}$ is in bijection
with a nonempty subset $I\subseteq\rows$. Let $\fmixed_{I}$ denote
the cells of $\fmixed$ contained in the coordinate subspace $\RR^{I}\times\{0\}^{\overline{I}}$
of $\RR^{d}$. For $S\in\{-1,0,1\}^{d}\smallsetminus\zero$ and $\sigma\in\fmixed_{\supp(S)}$,
let $\sigma_{S}:=\cube_{S}+S\cdot\sigma$.

\begin{proposition}

\label{prop:patchworking complex is PL sphere}

Define the collection of polytopes given by
\[
\octfmixed:=\left\{ \sigma_{S}:S\in\left\{ -1,0,1\right\} ^{d}\smallsetminus\zero,\;\sigma\in\fmixed_{\supp(S)}\right\} .
\]
Then the following statements hold:
\begin{enumerate}
\item $\octfmixed$ is a polyhedral cell complex which subdivides
the boundary of $\cube_{d}+n\crosspolytope_{d}$. 
\item For each $\sigma_{S}\in\octfmixed$, both $S$ and $\sigma$
can be recovered from $\sigma_{S}$. 
\item For $\sigma_{S},\tau_{T}\in\octfmixed$, we have $\sigma_{S}\subseteq\tau_{T}$
if and only if $S\geq T$ and $\sigma\subseteq\tau$.
\end{enumerate}
\end{proposition}

\begin{remark}The subdivision $\fmixed$ in the statement of Proposition~\ref{prop:patchworking complex is PL sphere} can be replaced by any polyhedral subdivision of $n\ssimplex_{d-1}$.
\end{remark}
\begin{proof}

First note that (2) follows from the fact that each $\sigma_{S}=\cube_{S}+S\cdot\sigma$
is a Minkowski sum of two affinely independent polytopes. Therefore,
projection allows us to recover both $\cube_{S}$ and $S\cdot\sigma$,
and therefore the pair $(S,\sigma)$. 

Recall the general fact that $F$ is a proper face of the Minkowski
sum $K+L$ of two full-dimensional polytopes $K$ and $L$ if and
only if there exists a non-zero objective function $\cc$ such that
$F=K_{\cc}+L_{\cc}$, where $K_{\cc}$ and $L_{\cc}$ denote the faces
of $K$ and $L$, respectively, maximized by $\cc$. Specializing
to the case $K=\cube_{d}$ and $L=n\crosspolytope_{d}$, we have $K_{\cc}=\cube_{S}$
and $L_{\cc}=S\cdot(n\ssimplex_{I})$, where $S$ is the componentwise
sign vector of $\cc$, $I$ is the set of all $i\in\rows$ such that
$\left|\cc_{i}\right|=\max_{k\in\rows}\left|\cc_{k}\right|$, and
$\ssimplex_{I}:=\conv(\ee_{i}:i\in I)$. It follows that the collection
of proper faces of $\cube_{d}+n\crosspolytope_{d}$ is given by
\[
\left\{ \cube_{S}+S\cdot(n\ssimplex_{I})\;:\;S\in\left\{ -1,0,1\right\} ^{d}\smallsetminus\zero,\;\emptyset\subsetneq I\subseteq\supp(S)\right\} .
\]
Since $\cube_{S}+S\cdot(n\ssimplex_{I})$ is the union of the cells
$\left\{ \sigma_{S}:\sigma\in\fmixed_{I}\right\} $, this shows that
the cells in $\octfmixed$ cover the boundary of $\cube_{d}+n\crosspolytope_{d}$.
The above fact about faces of Minkowski sums can also be used to show
that the faces of $\sigma_{S}=\cube_{S}+S\cdot\sigma$ are given by
$\left\{ \tau_{T}:\tau\text{ face of }\sigma,\;T\supseteq S\right\} $.
This establishes (3), and that $\octfmixed$ is closed under
taking faces.

To establish (1), it remains to show that the intersection of two
intersecting cells of $\sigma_{S},\tau_{T}\in\octfmixed$
is a face of both. If $\sigma_{S}\cap\tau_{T}$ is non-empty, then
$S$ and $T$ are conformal sign vectors since otherwise, if (say)
$i\in S^{+}\cap T^{-}$, then $\sigma_{S}$ would lie in the halfspace
$x_{i}\geq1$, while $\tau_{T}$ lies in the halfspace $x_{i}\leq-1$.
Now, we would like to show $\sigma_{S}\cap\tau_{T}=\left(\sigma\cap\tau\right)_{S\circ T}$,
where $S\circ T$ denotes sign vector composition. As argued above,
$\left(\sigma\cap\tau\right)_{S\circ T}$ is a face of both $\sigma_{S}$
and $\tau_{T}$. Thus, it remains to show that $\sigma_{S}\cap\tau_{T}$
is contained in $\left(\sigma\cap\tau\right)_{S\circ T}$.

Suppose $u+s=v+t$, where $u\in\cube_{S}$, $v\in\cube_{T}$, $s\in S\cdot\sigma$,
$t\in T\cdot\tau$. We are done if we can show $u=v$ and $s=t$.
For this it suffices to show $u_{i}=v_{i}$ for all $i\in\rows$.
Since $S,T$ are conformal, we have $u_{i}=v_{i}$ for all $i\in\supp(S)\cap\supp(T)$.
For $i\in\rows\smallsetminus(\supp(S)\cup\supp(T))$, we have $s_{i}=t_{i}=0$,
and hence $u_{i}=v_{i}$. Thus it remains to show $u_{i}=v_{i}$ in
the case when $i\in\supp(S)\smallsetminus\supp(T)$ or $i\in\supp(T)\smallsetminus\supp(S)$.

Suppose $i\in\supp(S)\smallsetminus\supp(T)$. The fact $i\in\supp(S)$
implies $\left|u_{i}\right|=1$, and the fact $i\notin\supp(T)$ implies
$t_{i}=0$. Now $u_{i}s_{i}\geq0$, which implies
\[
1+\left|s_{i}\right|=\left|u_{i}+s_{i}\right|=\left|v_{i}+t_{i}\right|=\left|v_{i}\right|\leq1.
\]
Hence $s_{i}=0=t_{i}$, and so $u_{i}=v_{i}$. The case $i\in\supp(T)\smallsetminus\supp(S)$
is proven analogously. \end{proof}

Recall that, by the Cayley trick, the cells in $\fmixed$ encode the simplices in $\mathcal{T}$,
for which no node in $\ground$ is isolated. This directly shows that
they fulfill property (E1) and (E2) of elimination systems (Definition~\ref{def:elimination+systems}). A proof that $\fmixed$ satisfies (E3) can be found in \cite[Proposition~4.12]{OhYoo:2011},
and this result has been generalized to arbitrary mixed subdivisions
in ~\cite[Theorem~7.11]{Horn1}.
Hence, we conclude:

\begin{proposition}\label{prop:toms_are_elimination_systems}

The set of forests encoded by $\fmixed$ form an
elimination system.\qed

\end{proposition}

Observing that $\sigma_{F}\in\fmixed_{I}$ if and only if $\supp_{\rows}(F)\subseteq I$,
and that on the level of posets, taking the dual just amounts to reversing
the ordering, we have

\begin{corollary}\label{cor:isomorphism_from_face_poset_to_combinatorial_poset}

The map $(\sigma_{F})_{S}^{\dual}\mapsto(S,F)$ determines an isomorphism
from the face poset of the dual complex $\octfmixed^{\dual}$
to the poset $\poset(\fmixed)$ defined in Definition~\ref{def:elim_poset}. 

\end{corollary}

\label{subsec:elimination_systems_quotients}

\subsection{The Map $\varphi:\protect\poset(\protect\fmixed)\protect\eqrel_{A}\;\rightarrow\protect\lattice(\widetilde{\protect\orientedmatroid})$}

We next consider the labeling of the elements of $\poset(\fmixed)\eqrel_{A}$ by sign vectors.
For this, we use the connection between the pairs $(S,F)$ denoting cells of the extended patchworking complex and covectors established in Corollary~\ref{coro:more+general+covectors}.

Now, we look at the particular elimination system given by the fine mixed subdivision $\fmixed$. 
In the following proposition,
let $\lattice(\widetilde{\orientedmatroid})$ denote the poset of
non-zero covectors of $\widetilde{\orientedmatroid}$. Let $\poset(\fmixed)\eqrel_{A}$
be the poset as in Section \ref{subsec:Quotients_of_P(S)}.

Recall from Section~\ref{subsec:Quotients_of_P(S)} the system of mixed signs $\left\{ 0,+,-,\pm\right\}$.
Using this as an intermediate step, one sees that the following map extending the map $\psi_A$ of Definition~\ref{def:triang+to+covectors+OM} is well-defined on its equivalence classes. 

\begin{definition} \label{def:quotient+to+OM}

Define the map $\posmap:\poset(\fmixed)\eqrel_{A}\,\rightarrow\{-1,0,1\}^{\widetilde{\ground}}$ by $\posmap_A(S,F)=(S,\psi_A(S,F))$. 

\end{definition}

As we fix $A$ most of the time, we just set $\posmap(S,F) = \posmap_A(S,F)$. 

\begin{example}[{Ex.~\ref{ex:quotients+signed+forests} continued}]
Recall that we could identify the equivalence classes of the four sign
vectors $S_{\ell}A_{F_{\ell}}$ for $\ell\in[4]$ with $(\pm,-,+)$,
$(\pm,-,+)$, $(\pm,-,\pm)$ and $(\pm,-,-)$. This shows that the
images of the three equivalence classes of $(S_{\ell},F_{\ell})$ (as $(S_{1},A_{F_{1}})\sim_{A}(S_{2},A_{F_{2}}$))
are the sign vectors 
\[
(-,+,+,0,-,+),(-,+,+,0,-,0),(+,+,+,0,-,-).
\]
\end{example}

Note that a similar map was used in~\cite[\S 6]{Horn1} to prove the
representation theorem for tropical oriented matroids. 

%We show the next claim by reducing it to the results from Section~\ref{sec:topes+covectors}.
  Since $\varphi$ is constant on the equivalence classes of $\eqrel_{A}$, we just fix an element $(S,F) \in \poset(\fmixed)$.
  With this, we associate the bipartite graph $T$ on $\rows \sqcup \widetilde{\ground}$ having the edges of $F$ and edges between the nodes of $\rows$ and its copy $\widetilde{\rows}$ within $\widetilde{\ground}$ for each element $\rows$ in the support of $S$.
  Now, the claim follows from Proposition~\ref{prop:covectors+from+topes}. 

\begin{corollary}
  For all $(S,F)\in\poset(\fmixed)\eqrel_{A}$,
we have $\varphi(S,F)\in\lattice(\widetilde{\orientedmatroid})$.
\end{corollary}

\begin{proposition} The map $\varphi\colon\poset(\fmixed)\eqrel_{A}\;\rightarrow\lattice(\widetilde{\orientedmatroid})$
is a poset map.

\end{proposition}

\begin{proof}

Suppose $(S,F)\leq(S',F')$ in $\poset$. Let $(S,X)=\varphi(S,F)$,
and let $(S',X')=\varphi(S',F')$. Then $S\leq S'$, and because $F\supseteq F'$,
the passage from $SA_{F}$ to $SA_{F'}$ only decreases the number
of non-zero entries in each column of $SA_{F}$. However $SA_{F'}$
still has at least one non-zero entry in each column. From this we
conclude $X\leq X'$, and therefore $\posmap$ respects order.\end{proof}

Recall that a poset $\poset$ is a \emph{sphere} if its order complex is a sphere; see Definition~\ref{order_complex}.

\begin{theorem}[Borsuk--Ulam]\label{thm:Borsuk_Ulam}

Let $\poset,\qoset$ be posets such that both are homeomorphic to
$S^{d-1}$ and both are equipped with a fixed-point free involutive
automorphism $x\mapsto-x$. Let $\posmap:\poset\rightarrow\qoset$
be a poset map satisfying $\posmap(-x)=-\posmap(x)$ for all $x\in\poset$.
Then $\posmap$ is surjective.

\end{theorem}

\begin{corollary}\label{cor:poset_map_isomorphism}Assuming $\poset(\fmixed)\eqrel_{A}$ is a $(d-1)$-sphere, the map $\varphi:\poset(\fmixed)\eqrel_{A}\rightarrow\lattice(\widetilde{\orientedmatroid})$
is an isomorphism.

\end{corollary}

\begin{remark}
We show that $\poset(\fmixed)\eqrel_{A}$ is indeed a $(d-1)$-sphere in Section~\ref{sec:putting_it_all_together}.
\end{remark}

\begin{proof}

For $(S,F)\in\poset(\fmixed)\eqrel_{A}$, the interpretation of $SA_{F}\eqrel$
as a generalized sign vector in $\{0,-,+,\pm\}^{\Pi}$ shows that
$\varphi$ is injective. To see that $\varphi$ is surjective, we
simply note that $(S,F)\mapsto-(S,F):=(-S,F)$ is a fixed-point free
involutive automorphism of $\poset(\fmixed)$ which descends to $\poset(\fmixed)\eqrel_{A}$,
while $X\mapsto-X$ is one of $\lattice(\widetilde{\orientedmatroid})$.
Furthermore, by definition of $\varphi$, we have $\varphi(-S,F)=(-S,-X)=-(S,X)=-\varphi(S,F)$.
As $\widetilde{\orientedmatroid}$ has rank $d$, the poset $\lattice(\widetilde{\orientedmatroid})$ is a $(d-1)$-sphere by Theorem~\ref{thm: arrangement of pseudo from order cpx}, and hence the conclusion follows from the Borsuk-Ulam Theorem.\end{proof}

\subsection{Pseudosphere Arrangements from Regular Cell Complexes}

The following result shows how to get pseudosphere arrangements from
regular cell complexes:

\begin{theorem}[{\cite[Theorem 4.3.3, Proposition 4.3.6]{BLSWZ:1993}}]

\label{thm: arrangement of pseudo from order cpx}Let $\orientedmatroid$
be an oriented matroid of rank $d$ on the ground set $E$. Let $\Delta$ be a
regular cell complex with face poset $\poset$, such that there is
a poset isomorphism $\poset\simeq\covectors(\orientedmatroid)$. Thus
each cell $\sigma_{X}$ of $\Delta$ is labeled by some non-zero
covector $X$ of $\orientedmatroid$. For each $k\in E$, define the
subcomplex
\[
\Delta_{k}:=\left\{ \sigma_{X}\in\Delta:X_{k}=0\right\} .
\]
 Then $\norm{\Delta}$ is a $(d-1)$-sphere, and the spaces $\norm{\Delta_{k}}$
ranging over all $e\in E$ form an arrangement of pseudospheres within
$\norm{\Delta}$ representing $\orientedmatroid$.

\end{theorem}

\begin{remark}

This theorem is really the uniqueness assertion of \cite[Theorem 4.3.3]{BLSWZ:1993},
whose proof can be traced back to \cite[Proposition 4.7.23]{BLSWZ:1993}.\end{remark}

\subsection{Putting it all Together}

\label{sec:putting_it_all_together}With all the pieces now in place, we are ready to prove Theorem \ref{thm:arrangement_of_pseudospheres_extended},
which asserts that our patchworking procedure yields a pseudosphere
representation of $\widetilde{\orientedmatroid}$.

\begin{proof}[Proof of Theorem \ref{thm:arrangement_of_pseudospheres_extended}]

As shown for Proposition~\ref{prop:toms_are_elimination_systems},
a fine mixed subdivision $\fmixed$ of $n\ssimplex_{d-1}$ gives rise
to an elimination system as in Definition \ref{def:elimination+systems}.
Abusing notation, we denote this elimination system also by $\fmixed$.
We let $\poset(\fmixed)$ be the poset of $\fmixed$ obtained by introducing
signs as in Definition~\ref{def:elim_poset}. 

Let $\poset(\Delta)$ denotes the face poset of $\Delta=\octfmixed^{\dual}$.
By Corollary \ref{cor:isomorphism_from_face_poset_to_combinatorial_poset},
we have $\poset(\Delta)\simeq\poset(\fmixed)$. Hence the poset quotient
$\poset(\fmixed)\eqrel_{A}$ induces a quotient $\poset(\Delta)\eqrel_{A}$.
By Theorem \ref{thm:elim_system_poset_admits_factorization}, then,
$\poset(\Delta)\eqrel_{A}$ admits a factorization $\poset(\Delta)=\poset_{0},\poset_{1},\ldots,\poset_{k}=\poset(\Delta)\eqrel_{A}$
into elementary quotients, such that the augmented poset $\lattice(\poset_{i})$
is a lattice for each $i=0,1,\ldots,k-1$.

As a polyhedral complex on the boundary of a $d$-dimensional polytope, $\octfmixed$
is a PL $(d-1)$-sphere. Hence, by Proposition \ref{prop:dual cell cpx}, $\Delta=\octfmixed^{\dual}$
is also a PL $(d-1)$-sphere. In particular, by Proposition \ref{prop:pl_sphere_cells_are_pl_balls},
each cell $\sigma^{\dual}$ in $\Delta$ is a PL ball. It follows,
by Corollary \ref{cor:factorization_implies_quotient_complex_is_regular},
that $\Delta\eqrel_{A}$ is a regular cell complex with face poset
$\poset(\Delta)\eqrel_{A}$. In particular, $\poset(\Delta)\eqrel_{A}$ is a $(d-1)$-sphere.

By Corollary \ref{cor:poset_map_isomorphism}, we have isomorphisms
$\poset(\Delta)\eqrel_{A}\,\simeq\poset(\fmixed)\eqrel_{A}\,\simeq\lattice(\widetilde{\orientedmatroid})$.
For $k\in\widetilde{\ground}$, define the subcomplex
\[
(\Delta\eqrel_{A})_{k}:=\left\lbrace\uni(\sigma_{(S,F)}^{\dual}\eqrel_{A})\in\Delta\eqrel_{A}\;:\,\varphi(S,F)_{k}=0\right\rbrace
\]
of $\Delta\eqrel_{A}$.
Now, Theorem \ref{thm: arrangement of pseudo from order cpx} implies that the spaces $\norm{(\Delta\eqrel_{A})_{k}}$ ranging over
all $k\in\widetilde{\ground}$ form an arrangement of pseudospheres
within $\norm{\Delta\eqrel_{A}}=\norm{\Delta}$ representing~$\widetilde{\orientedmatroid}$.

It remains to show that $\norm{(\Delta\eqrel_{A})_{k}}=\norm{\Delta_{k}}$
for all $k\in\widetilde{\ground}$, where $\Delta_{k}$ is a subcomplex of $\Delta$ defined in (\ref{eq:Delta_i}) and (\ref{eq:Delta_j}). For this note that the closed cells
of $\Delta\eqrel_{A}$, and hence all subcomplexes of $\Delta\eqrel_{A}$,
each consist of a union of members of~$\Delta$. Hence, it suffices
to show that for all $\sigma^{\dual}\in\Delta$ and $k\in\widetilde{\ground}$,
we have $\sigma^{\dual}\subseteq\norm{(\Delta\eqrel_{A})_{k}}$ if
and only if $\sigma^{\dual}\in\Delta_{k}$.

For $\sigma^{\dual}\in\Delta$ and $k\in\widetilde{E}$, we have $\sigma^{\dual}\subseteq\norm{(\Delta\eqrel_{A})_{k}}$
if and only if there exists $(S,F)\in\poset(\fmixed)$ such that $\varphi(S,F)_{k}=0$
and
\[
\sigma^{\dual}\subseteq\bigcup_{(S,G)\,\sim_{A}\,(S,F)}\sigma_{(S,G)}^{\dual}.
\]
As $\Delta$ is a regular cell complex, the interiors of the balls
in $\Delta$ are disjoint, and so the above containment holds true
if and only if $\sigma^{\dual}=\sigma_{(S,F)}^{\dual}$ for some $(S,F)\in\poset(\fmixed)$
such that $\varphi(S,F)_{k}=0$. If $k\in\rows$, then $\varphi(S,F)_{k}=0$
iff $S_{k}=0$. If $k\in\ground$, we have $\varphi(S,F)_{k}=0$ if
and only if there exist $(i,k),(\ell,k)\in F$ such that $S_{i}A_{i,k}=-S_{\ell}A_{\ell,k}\neq0$.
In either case, we conclude $\sigma^{\dual}\subseteq\norm{(\Delta\eqrel_{A})_{k}}$
if and only if $\sigma^{\dual}\in\Delta_{k}$.\end{proof}

\medskip

{\sc Acknowledgement.} The first author was funded by the Deutsche Forschungsgemeinschaft (DFG, German Research Foundation) under Germany's Excellence Strategy -- The Berlin Mathematics Research Center MATH+(EXC-2046/1, project ID: 390685689).
He is grateful to Josephine Yu for helpful discussions on patchworking, in particular pointing out the reference \cite{Sturmfels:1994a}.
This project has received funding from the European Research Council (ERC) under the European Union's Horizon 2020 research and innovation programme (grant agreement n° ScaleOpt-757481).
%The second author was supported by the European Research Council (ERC) Starting Grant ScaleOpt, No.~757481.
The second author also thanks Xavier Allamigeon and Mateusz Skomra for inspiring discussions and Ben Smith for helpful comments.
The third author was supported by Croucher Fellowship for Postdoctoral Research and Netherlands Organisation for Scientific Research Vici grant 639.033.514 during his affiliation to Brown University and University of Bern, respectively.
He also thanks Cheuk Yu Mak for pointing out \cite{RubermanStarkton:2019}.
The second and third author both thank the {\em Tropical Geometry, Amoebas, and Polyhedra} semester program at Institut Mittag-Leffler for the introduction that started the project.
The authors thank Laura Anderson and Jesus De Loera for comments on our work.

% -------------------------------- bibliography-----------------------------
\bibliographystyle{amsplain}
\bibliography{pmf}

\end{document}